\documentclass[11pt]{amsart}
\usepackage{amsmath,amssymb}
\usepackage{tikz}
\usepackage{hyperref}
\newtheorem{theorem}{Theorem}

\newtheorem{proposition}{Proposition}
\newtheorem{lemma}{Lemma}

\newcommand{\sm}{\left(\begin{smallmatrix}}
\newcommand{\esm}{\end{smallmatrix}\right)}
\theoremstyle{definition}

  \DeclareMathAlphabet{\newcal}{U}{dutchcal}{m}{n}
  \DeclareMathAlphabet{\eucal}{U}{eus}{m}{n}

\def\Q{\mathbb Q} \def\R{\mathbb R} \def\C{\mathbb C} \def\Z{\mathbb Z}  \def\H{\mathcal H}
\def\={\;=\;} \def\+{\,+\,} \def\m{\,-\,}  \def\df{\;:=\;}  \def\la{\langle} \def\ra{\rangle}  
\def\e{\varepsilon}  \def\t{\tau}  \def\g{\gamma} \def\G{\Gamma}  \def\D{\Delta} \def\T{\Theta}  \def\j{V}
\def\V{\mathbb V} \def\hV{\widehat{\V}}  \def\W{\mathbb W}  \def\L{\newcal L} 
\def\CC{\mathbf C} \def\BC{\mathbf B}     \def\Cg{\mathfrak C}    \def\Rg{\mathfrak R}  \def\Dg{\mathfrak D} 
\def\BB{\mathcal B}  \def\BBc{\BB^{\rm cusp}}    
    \def\r{\mathbf r}
\def\be{\begin{equation}}  \def\ee{\end{equation}}  \def\bes{\begin{equation*}}  \def\ees{\end{equation*}}
\def\ba{\begin{align}}  \def\bas{\begin{align*}}  \def\ea{\end{align}}  \def\eas{\end{align*}}

\theoremstyle{remark}

\begin{document}

\title[Periods of modular forms on $\Gamma_0(N)$]
  {Periods of modular forms on $\Gamma_0(N)$ \\ and products of Jacobi theta functions}
\author{YoungJu Choie, Yoon Kyung Park, and Don Zagier}
 
\address{YoungJu Choie \endgraf  Department of Mathematics\endgraf
Pohang University of Science and Technology \endgraf Pohang, Republic of Korea}
\email{yjc@postech.ac.kr}

\address{Yoon Kyung Park\endgraf  Institute of Mathematical Sciences\endgraf
Ewha Womans University
\endgraf   Seoul, Republic of Korea}
\email{ykp@ewha.ac.kr}

\address{Don Zagier\endgraf  Max Planck Institute for Mathematics \endgraf
Bonn, Germany}  \email{dbz@mpim-bonn.mpg.de}



\begin{abstract}  Generalizing a result of~\cite{Z1991} for modular forms of level~one,
we give a closed formula for the sum of all Hecke eigenforms on $\Gamma_0(N)$, multiplied 
by their odd period polynomials in two variables, as a single product of Jacobi theta series 
for any squarefree level $N$.  We also show that for $N=2$,~3 and~5 this formula completely 
determines the Fourier expansions of all Hecke eigenforms of all weights on $\Gamma_0(N)$.
\end{abstract}

\maketitle
 
\section{Introduction and statement of theorems} \label{Intro}

It is well known that to any cusp form~$f$ on the full modular group $\Gamma_1=SL_2(\mathbb{Z})$
one can associate the {\em period polynomial} 
 \be\label{rfdef}
   r_f(X) \df \int_0^\infty f(\t)\,(X-\t)^{k-2}\,d\t \qquad(\text{$k$ = weight of~$f$})  \ee
and that the maps $f\mapsto r_{f}^{\rm ev}$ and $f \mapsto r_{f}^{\rm od}$ assigning to~$f$ the even 
and odd parts of~$r_f$ are both injective, with known images.  Furthermore, if $f$ is a 
Hecke form (= normalized Hecke eigenform), then the odd two-variable polynomial
  \be\label{Rfdef}  R_f(X,Y) \df \frac{r_f^{\rm ev}(X)\,r_f^{\rm od}(Y) \+ r_f^{\rm od}(X)\,r_f^{\rm ev}(Y)}
  {(2 i)^{k-3} \,\langle f,\, f\rangle}  \quad\in\;\C[X,Y] \ee 
transforms under $\sigma\in\text{Gal}(\C/\Q)$ by $R_{\sigma(f)}=\sigma(R_f)$,
so $R_f$ has coefficients in the number field~$\Q_f$ generated by the Fourier coefficients of~$f$ and, 
denoting by $\BBc_k$ the basis of Hecke cusp forms~$f$ of weight~$k$ on~$\G_1$, the sum
\be\label{Zmain2}   C_k^{\rm cusp}(X,Y,\t) \df \frac1{(k-2)!}\,\sum_{f\in\BBc_k} R_f(X,Y)\,f(\t) \ee
belongs  to $\Q[[q]][X,Y]$  for each integer~$k>0$. (Here $q=e^{2\pi i\t}$ as usual.)

In~\cite{Z1991}, a surprising identity was proved showing that all of these functions can be assembled 
into a single four-variable generating function that has a very simple expression as a 
product of Jacobi theta functions.  More precisely, it was shown that if one defines $C_k(X,Y,\t)$ for $k>0$ 
by adding to $C_k^{\rm cusp}(X,Y,\t)$ a term for the normalized Eisenstein series $G_k$ of weight~$k$ (with 
a suitable definition of $R_{G_k}(X,Y)$ that will be recalled below), then the generating function 
\be\label{genfn1} \CC(X,Y,\t,T) \df \frac{(X+Y)(XY-1)}{X^2Y^2T^2}\+\sum_{k=2}^\infty C_k(X,Y,\t)\,T^{k-2}\ee 
is given by the formula
\be\label{Zmain3}
\CC(X,Y,\t,T) \= \theta'_\t(0)^2\,\frac{\theta_\t\bigl((XY-1)T \bigr)\, \theta_\t\bigl((X+Y)T\bigr) } 
{\theta_\t\bigl(XYT \bigr)\,\theta_\t\bigl(XT\bigr)\ \theta_\t\bigl(YT\bigr)\,\theta_\t\bigl(T\bigr) }\;, 
\ee
where $\theta_\t(u)$ denotes the classical Jacobi theta function
\be\label{thetadef} \theta_\t(u) \= q^{1/8} e^{u/2} \prod_{n=1}^{\infty} (1-q^n)(1-q^ne^u)(1-q^{n-1}e^{-u})\,. \ee

Until now, attempts to find an analogous result for higher levels were not successful, and it was thought that 
the fact that the generating function for all level one Hecke forms and their periods had a {\it multiplicative} 
expression as a single product of theta functions was an accident, due to the special structure of~$\G_1$ 
as a group with two generators and of the algebra of modular forms on~$\G_1$ as a free algebra on two
generators. In this paper, we will show that this belief was wrong and that there is a statement of
precisely the same form for modular forms on $\Gamma_0(N)$ for every squarefree integer~$N$.
\begin{theorem}\label{Thm1} Let $N>0$ be squarefree and for every integer~$k\ge2$ define
\be\label{Main1}  C_{k,N}(X,Y,\t) \df \frac1{(k-2)!}\,\sum_{f\in\BB_{k,N}} R_f(X,Y)\,f(\t)\,,  \ee
where $\BB_{k,N}$ is the basis of Hecke forms on~$\G_0(N)$ defined in~\S2 and where $R_f(X,Y)$ is defined 
by~\eqref{rfdef} and~\eqref{Rfdef} when~$f$ is a cusp form and in the way explained in~\S2 when~$f$ 
is an Eisenstein series. Then the generating function
\be \label{Defgenfn} \CC_N(X,Y,\t,T) \df \frac{(X+Y)(NXY-1)}{NX^2Y^2T^2}
   \+\sum_{k=2}^\infty C_{k,N}(X,Y,\t)\,T^{k-2}  \ee
is given in terms of the level one theta function~\eqref{thetadef} by
\be\label{mm} 
 \CC_N(X,Y,\t,T) \= \frac{\theta_{\t}'(0)\,\theta_\t\bigl((X+Y)T\bigr)}{\theta_{\t}(XT)\,\theta_{\t}(YT)}
 \cdot\frac{\theta_{N\t}'\bigl(0)\,\theta_{N\t}\bigl(( NXY-1)T\bigr)}{\theta_{N\t}(NXYT)\,\theta_{N\t}(T)\,}\,. \ee
\end{theorem}

This theorem gives a simple way to find the polynomial $R_f(X,Y)$ in $\overline\Q[X,Y]$ for any 
Hecke cusp form~$f$ of weight~$k$ on~$\G_0(N)$ if the Hecke basis~$\BB_{k,N}$ is supposed known, just by
decomposing the coefficient of $T^{k-2}$ in the theta-product in~\eqref{mm} with respect to that basis and taking the 
coefficient of~$f$. In~\cite{Z1991} it was shown that for $N=1$ a stronger statement holds: even though the 
construction of the generating function involves summing over all Hecke forms in each weight and therefore apparently
has destroyed information about the individual terms, an easy algebraic lemma implies that the coefficient of~$T^{k-2}$
in the symmetrized generating function $\CC(X,Y,\t,T)+\CC(-X,Y,\t,T)$ has a {\it unique} decomposition as a sum of 
$\,\dim M_k(\G_1)\,$ terms of the form~$f(\t)r(X)s(Y)$ with $r(X)$ even and~$s(Y)$ odd, and therefore the formula uniquely
determines the Hecke eigenforms on~$\G_1$ themselves, as well as their period polynomials. The corresponding statement 
cannot be true for the generating function of Theorem~\ref{Thm1} in general, because for all but a few levels the space
of modular forms of weight~$k$ on~$\G_0(N)$ has a larger dimension than the space of polynomials of degree~$\le k-2$.  
(We plan to return to this question in a later paper, where Theorem~\ref{Thm1} is generalized to include more period 
polynomials for each eigenform.)  However, for small~$N$ this obstruction is not present and one can in fact prove
the exact analogue of this second result also:
\begin{theorem}\label{Thm2} For $N\in\{2,3,5\}$ the identity in Theorem~\ref{Thm1} completely
determines the Fourier coefficients and periods of all Hecke forms on $\Gamma_0(N)$.
\end{theorem}

The paper is organized as follows.  In Section~\ref{PerPolys} we give the precise definitions and main 
properties of all of the quantities occurring in Theorem~\ref{Thm1} and~\ref{Thm2}, including the definitions 
of the Petersson scalar product $\la f,f\ra$ and of the period polynomial $r_f(X)$ when $f$ is an Eisenstein series.
In the next section we compute the Eisenstein part of the generating function~\eqref{Main1} and use this to show
that the left- and right-hand sides of~\eqref{mm} agree at all cusps of~$\G_0(N)$.  This reduces the proof of
Theorem~\ref{Thm1} to the statement that the scalar products of any $f\in\BB_{k,N}^{\rm cusp}$ with the two
sides of~\eqref{mm} agree. In Section~\ref{Level1} we review the notion of Rankin-Cohen brackets and the formula 
expressing the periods of a cusp form in terms of its scalar products with Rankin-Cohen brackets of Eisenstein series,
and give a simplified presentation of the proof of~\eqref{Zmain3} given in~\cite{Z1991}. The generalizations of these
results to all squarefree levels~$N$ are then given in the following section, completing the proof of Theorem~\ref{Thm1}.
Along the way we also obtain a formula for the eigencomponents of $C_{k,N}$ with respect to the action of the
group of Atkin-Lehner operators (Theorem~\ref{Thm3}, Section~\ref{Eis})  
In Section~\ref{Examples} we give numerical examples illustrating Theorem~\ref{Thm1} for small levels and also showing 
both how Theorem~\ref{Thm2} works for the levels stated and how it fails for modular forms of level~7.  The proof of 
Theorem~\ref{Thm2}, together with some further information about modular forms of these small levels, is given 
in the final section.  

This paper is based on an earlier paper by the first two authors alone, completed in~2014, in which a result 
equivalent to Theorem~\ref{Thm1} was proved for prime levels.  It was then found by the third author that 
after some reorganization the same proof worked for arbitrary squarefree levels.

\section{Modular forms on $\G_0(N)$ and their periods}  \label{PerPolys}

In this section we give the complete definitions of all of the quantities appearing in the statements of 
the theorems given in the introduction, including the canonical basis of Hecke forms for~$\G_0(N)$
and the definition of the Eisenstein part of the generating series~\eqref{genfn1} and~\eqref{Defgenfn}.
None of the material is new, but we have included full details for the reader's convenience.

Our notations for modular forms are standard.  We write~$\H$ for the upper half-plane, 
$q=e^{2\pi i\tau}$ for a generic variable $\tau\in\H$, and $f(\t)=\sum a_n(f)\,q^n$ for the Fourier
expansion of a modular form~$f$. For~$k\in2\Z$ we define the ``slash operator" $|_k$ by
$$(f|_{k} \gamma)(\tau):=(ad-bc)^{k/2} (c\tau+d)^{-k}f(\frac{a\tau+b}{c\tau+d})\,, $$ 
for $\gamma = \pm\sm a & b\\ c& d\esm \in PGL_2^+(\mathbb{R})$ and for any function
$f:\H\rightarrow\C$.  Thus a modular form of weight~$k$ on a subgroup $\G\subset PGL_2^+(\mathbb{R})$ 
satisfies $f|_k\g=f$ for \hbox{all~$\g\in\G$}.

We denote~by $M_k$ and $S_k$ the spaces of modular forms and cusp forms, 
respectively, on the full modular group, and by $M_{k,N}$ or $M_k(\Gamma_0(N))$ and $S_{k,N}$ or 
$S_k(\Gamma_0(N))$ the spaces of modular forms and cusp forms, respectively, on the group $\Gamma_0(N)$, 
where $N$ will always denote a squarefree number and where $\G_0(N)$ (and later its normalizer $\G_0^*(N)$)
will always be considered as subgroups of $PGL_2^+(\mathbb{R})$.  The space $M_{k,N}$ has three main decompositions:
into Eisenstein series and cusp forms, into new forms and various types of old forms, and into eigenspaces 
for the Atkin-Lehner involutions.  Each of them has a particularly simple form because $N$ is squarefree. 

We start with the decomposition $M_{k,N}=M_{k,N}^{\rm Eis}\oplus S_{k,N}$, where $M_{k,N}^{\rm Eis}$
is the space spanned by the Eisenstein series of weight~$k$ and level~$N$.  The group~$\G_0(N)$ has 
$\nu(N)=2^t$ cusps, where $\nu(N)$ denotes the number of divisors of~$N$ and $t$ its number of prime 
factors. For $k\ge4$ the space $M_{k,N}^{\rm Eis}$ has the same dimension and a basis given by the functions
$G_k(d\t)$ with $d|N$, where $G_k=-\frac{B_k}{2k}+\sum_{n\ge1}\sigma_{k-1}(n)\,q^n$ is the normalized Eisenstein
series of level~1.  If~$k=2$, then these functions still span $M_{k,N}^{\rm Eis}$, but now this space has
dimension only $\nu(N)-1$, since the form $G_2(\t)$ is only quasimodular and the linear combination 
$\sum_{d|N}c_d\,G_2(d\t)$ is modular only if $\sum c_d/d=0$.

We next turn to old and new forms. By an {\it old form} on $\G_0(N)$ we mean any linear combination of functions 
$f(d\t)$ where $f$ is a modular form of level~$N_1$ with $N_1$ a strict divisor of~$N$ and $d$ a divisor of~$N/N_1$.  
Thus all Eisenstein series are old (and even ``very old," coming all the way from level~1) if $N>1$ and $k\ge4$.
We will consider them to be old also if $N>1$ and~$k=2$, since $G_2$ is quasimodular of level~1 and they
are therefore old as quasimodular forms as well as being modular of level~$N$. For $N>1$ the space of {\it new forms}
of weight~$k$ on~$\G_0(N)$ is a subspace of $S_{k,N}$ and is defined there as the space of forms that are orthogonal 
with respect to the Petersson scalar product to all old forms.  (For $N=1$ the new forms are simply the whole 
space~$M_k$.) We then have the decomposition
  $$  S_{k,N} \=   \bigoplus_{N_1|N} \bigoplus_{d|N/N_1} S_{k,N_1}^{\rm new}\bigl|_k\j_d\,, $$
where $\j_d=\sm \!d&0\\\!0&1\!\esm$, so that $f|_k\j_d(\t)=d^{k/2}f(d\t)$.  As in the
introduction, we define a {\it Hecke form} in $M_{k,N}$ to be a simultaneous eigenform~$f$ of all 
Hecke operators $T_n$ with $(n,N)=1$, normalized by~$a_1(f)=1$, in which case $f|_kT_n=a_n(f)f$
for all $n$ prime to~$N$.  In particular, $G_k$ is a Hecke form for $N=1$. The finite set $\BB_{k,N}^{\rm new}$ 
of Hecke forms in $M_{k,N}^{\rm new}$ (which are sometimes called the {\it newforms}, written with no space) 
forms a basis of this space. The set of all $f(d\t)$ with $f\in\BB_{k,N_1}^{\rm new}$ and $dN_1|N$
(with $d=1$ omitted and $f(d\t)$ replaced by $dG_2(d\t)-G_2(\t)$ if $k=2$ and~$N_1=1$) then forms
a basis of $M_{k,N}$.  But this is not a good choice, since its elements neither have multiplicative Fourier
coefficients nor are mutually orthogonal.  To get a better basis we must use the Atkin-Lehner operators.

Denote by $\Dg(N)$ the set of divisors of~$N$, made into a group isomorphic to $(\Z/2\Z)^t$ by the 
multiplication $N_1\star N_2=N_1N_2/(N_1,N_2)^2$. If $M\in\Dg(N)$, then since $(M,N/M)=1$ 
we can find $W_M\in\sm M\Z&\Z\\N\Z&M\Z\esm$ with \hbox{$\det(W_M)=M$}. The $W_M$ is called the Atkin-Lehner involution. Any two such matrices differ on the left
(and also on the right) by an element of~$\G_0(N)$, so the function $f|_kW_M$
for $f\in M_{k,N}$ is independent of the choice of matrix~$W_M$ and again belongs to~$M_{k,N}$.  This
defines an action of the group $\Dg(N)$ on $M_{k,N}$ and an eigenspace decomposition 
$M_{k,N}=\bigoplus_{\e\in\Dg(N)^\vee} M_{k,N}^\e$, where the sum ranges over the characters of $\Dg(N)$
(i.e., the homomorphisms $\e:\Dg(N)\to\{\pm1\}$) and where $M_{k,N}^\e$ is the space of $f\in M_{k,N}$ with
$f|_kW_M=\e(M)f$ for all $M|N$.  
This eigenspace decomposition is compatible with the splitting into new and old forms in two different
senses.  On the one hand, since the Atkin-Lehner involutions commute with the Hecke operators $T_n$ for 
\hbox{$(n,N)=1$}, every newform is automatically an eigenfunction of the group~$\Dg(N)$.  This means that the basis 
$\BB_{k,N}^{\rm new}$ of $M_{k,N}^{\rm new}$ is the union over all $\e\in\Dg(N)^\vee$ of the subset 
$\BB_{k,N}^{\rm new,\,\e}$ of newforms belonging to $M_{k,N}^\e$ (all of which are cusp forms except 
for $f=G_k$ in the case $N=1$, $\e=1$, $k\ge4$).  On the other hand, for each decomposition $N=N_1N_2$ and each 
character $\e_2\in\Dg(N_2)^\vee$, we have a linear map (cf.~\cite{CK2009}, Prop.~2)
  $$ \L_{k,N_2}^{\e_2}:\, M_{k,N_1}\,\to \,M_{k,N}, \;\quad
    \L_{k,N_2}^{\e_2}\bigl(M_{k,N_1}^{\e_1}\bigr)\;\subset\;M_{k,N}^{\e_1\e_2} \quad\bigl(\forall \e_1\in\Dg(N_1)^\vee\bigr) $$
given by
\be\label{defLkN}  \L_{k,N_2}^{\e_2}(f) \= f\,{\Bigl|}_{k}\Bigl(\sum_{d|N_2}\e_2(d)\,\j_d\Bigr)
 \=f\,{\Bigl|}_{k}\Bigl(\sum_{d|N_2}\e_2(d)\,W_d\Bigr) \,, \ee
where the second equality holds because $W_d^{\vphantom{-1}}\j_d^{-1}\in \G_0(N_1)$ if~$d|N_2$. It then follows
by induction on~$t$, the number of prime factors of~$N$, that
 $$ M_{k,n}^\e \= \bigoplus_{N=N_1N_2} \L_{k,N_2}^{\rho_2(\e)}\Bigl(M_{k,N_1}^{\rm new,\,\rho_1(\e)}\Bigr) 
  \quad\;\bigl(\e\in\Dg(N)^\vee,\; \rho_i(\e):=\e|_{\Dg(N_i)}\bigr)\,.$$
Together these two statements imply that $M_{k,N}^\e$ has a basis $\BB_{k,N}^\e$ given by
$$ \BB_{k,N}^\e \= \coprod_{N=N_1N_2}\Bigl\{ \L_{k,N_2}^{\rho_2(\e)}(f)\,\Big|\,
  f\in \BB_{k,N_1}^{\rm new,\,\rho_1(\e)}\Bigr\} $$
for every $\e\in\Dg(N)^\vee$.  The union of these bases for all $\e\in\Dg(N)^\vee$ is the basis $\BB_{k,N}$ of $M_{k,N}$
occurring in Theorem~\ref{Thm1}.  We have $\BB_{k,N}=\BB^{\rm Eis}_{k,N}{\scriptstyle\coprod}\BB^{\rm cusp}_{k,N}$ with 
$\BB^{\rm Eis}_{k,N}$ consisting of the functions $G_{k,N}^\e:=\L_{k,N}^\e(G_k)$ for all $\e\in\Dg(N)^\vee$ except 
$\e=1$ in the case~$k=2$.  We also observe that the group of Atkin-Lehner involutions~$W_M$ permutes the $2^t$ cusps 
of~$\G_0(N)$ simply transitively, and that the group $\G_0^*(N)$ generated by $\G_0(N)$ and all of the~$W_M$,
which is the normalizer of $\G_0(N)$ in $PGL_2^+(\R)$, has only one cusp. This will be used later.

For each Hecke form $f$ the $L$-series $L(f,s)=\sum_{n=1}^\infty a_n(f)\,n^{-s}$ has an Euler product 
expansion $L(f,s)=\prod_p L(f,p^{-s})_p$, where the product is over all primes and where each factor
$L(f,X)_p$ is a rational function of~$X$. If $f\in M_{k,N}^\e$ is a newform then these functions have the form
\be\label{newLfactor} L(f,X)_p \= \begin{cases} \bigl(1\m a_p(f)X\+p^{k-1}X^2\bigr)^{-1} & \text{ if $p\nmid N$}, \\
  \bigl(1\+\e(p)\,p^{k/2-1}\,X\bigr)^{-1}& \text{ if $p\mid N$}, \end{cases} \ee
while for an oldform $f=\L_{k,N_2}^{\e_2}(f_1)$ with $f_1\in\BB_{k,N_1}^{\rm new,\,\e_1}$ we have
\be\label{oldLfactor} L(f,X)_p \= L(f_1,X)_p\,\cdot\,\begin{cases} 1 & \text{ if $p\nmid N_2$,} \\
  1+\e_2(p)\,p^{k/2}\,X & \text{ if $p\mid N_2$.} \end{cases} \ee
(This includes the case when $N_1=1$ and $f_1=G_k$, in which case $L(f_1,X)_p=(1-X)^{-1}(1-p^{k-1}X)^{-1}$.)
Combining these statements gives a description of $L(f,X)_p$ for all $f\in\BB_{k,N}$ and all primes~$p$ that
will be used later.

Finally, we discuss, first in the case of cusp forms, the two quantities $\la f,f\ra$ and $r_f(X)$ appearing 
in the definition of $R_f(X,Y)$.  The former, of course, denotes the Petersson scalar product of~$f$ with itself, 
but we should emphasize that in our normalization this scalar product is defined as the integral of 
$|f(x+iy)|^2y^{k-2}dx\,dy$ over a fundamental domain for the group~$\G_1$ in the case of eq.~\eqref{Zmain2}
and for the group~$\G_0(N)$ in the case of eq.~\eqref{Main1}.  When we are discussing $\G_0(N)$, we shall
always use $\la\,\cdot\,,\,\cdot\,\ra$ to denote the scalar product with respect to that group, so that if
a form $f\in S_{k,N}$ happens to be modular on $\G_0(N_1)$ for some proper divisor $N_1$ of~$N$ then $\la f,f\ra$
is $[\G_0(N_1):\G_0(N)]$ times the scalar product of $f$ with itself with respect to~$\G_0(N_1)$, which to 
avoid confusion we then denote by~$\la f,f\ra_{N_1}$. If $f\in\BB_{k,N}^\e$ has the form $\L_{k,N_2}^{\e_2}(f_1)$ 
for some $f_1\in\BB_{k,N_1}^{\e_1}$, where $N=N_1N_2$ and $\e=\e_1\e_2$ as above, then the two scalar products
$\la f,f\ra=\la f,f\ra_N$  and $\la f_1,f_1\ra_{N_1}$ are related by
\be\label{PSPRatio}  \la f,f\ra \= \la f_1,f_1\ra_{N_1} \,\cdot 
  \prod_{p|N_2} 2\,\Bigl(p\+\e_2(p)\,a_p(f)\,p^{1-k/2}\+1\Bigr)\,, \ee
as one can show in several ways, e.g.~by applying the Rankin-Selberg method to express $\la f,f\ra$ as
a multiple of the residue of the Rankin-Selberg zeta function $\sum_n a_n(f)^2\,n^{-s}$ at $s=k$ and then
using the relationship~\eqref{oldLfactor} to relate that zeta function to the corresponding one for~$f_1$.

The period polynomial $r_f$ is defined for a cusp form $f$ of weight $k$ and any level by~\eqref{rfdef} and 
belongs to the space $\mathbb{V}_{k-2} := \la 1,X, \dots, X^{k-2}\ra$ of~$\C[X]$. We have
$r_f(X)=\sum_{n=0}^{k-2} (-1)^n \sm k-2\\n \esm r_n(f) X^{k-2-n}$ with ``periods" $r_n(f)$ defined by
  \be\label{rnf}   r_n(f) \df \int_0^\infty f(\tau)\,\tau^n\,d\tau \= \frac{n!\,i^{n+1}}{(2\pi)^{n+1}}L(f,n+1)
  \qquad　(0\le n\le k-2)\,. \ee
As in the introduction, we write
$r_f^{\rm ev}(X)$ and $r_f^{\rm od}(X)$ for the even and odd parts of $r_f(X)$ and write ``ev/od"
for statements that apply to both parities, e.g., $r_f^{\rm ev/od} \in \mathbb V_{k-2}^{\rm ev/od}$
with an obvious notation.  The group $PGL_2^{+}(\mathbb{R})$ acts on $\mathbb{V}_{k-2}$ by $|_{2-k}$.  
In particular, for the Fricke involution $W_N=\sm0&-1\\N&0\esm \in PGL_2^+(\R)$, a simple calculation shows 
that we have $r_{f|_kW_N}=-r_f|_{2-k}W_N$ (the minus sign arises because $W_N$ interchanges 0 and~$\infty$), so
\be f\,\in\,S_{k,N}^\e \quad \Rightarrow \quad r_f\,\in\,
  \V_{k-2,N}^\e \,:=\, \text{Ker}\bigl(1+\e(N)\, W_N,\V_{k-2}\bigr)\,. \ee
Equivalently, $r_{k-2-n}(f)=(-1)^{n+1}\e(N)N^{k/2-1-n}\,r_n(f)$, corresponding to the 
functional equation of the $L$-series of~$f$.  Another simple calculation shows that $r_{f|_kV_d}=r_f|_{2-k}V_d$
for all~$d|N$ (this time with a plus sign, because the matrix $V_d=\sm d&0\\0&1\esm$ fixed the endpoints~0
and~$\infty$ of the integral~\eqref{rfdef}).  Together with~\eqref{defLkN} this implies the relationship
\be\label{relPer}  f=\L_{k,N_2}^{\e_2}(f_1) \;\Rightarrow\; r_f(X)\= \sum_{d|N_2}\e_2(d)\,d^{1-k/2}\,r_{f_1}(dX)\ee
between the period polynomial of an old form and the period polynomial of the new form of lower level from which
it is induced.  Finally, for later purposes we mention that $r_f(X)$ for $f\in S_{k,N}^\e$ can also be given by
\be  \label{Newrfdef} 
 r_f(X) \=  \widetilde f(X,\t) \m \e(N)\,N^{k/2-1}\,X^{k-2}\,\widetilde f\Bigl(-\frac1{NX},\,-\frac1{N\t}\Bigr)\ee
for any $\t\in\H$, where $\widetilde f(X,\t)$ is the truncated version of~\eqref{rfdef} defined by
\be \label{rtildedef}\widetilde f(X,\t) \= \int_\t^\infty f(\t')\,(X-\t')^{k-2}\,d\t' \qquad(\t\in\H)\,.  \ee

If $f$ is a  Hecke form, then it is known that there are non-zero numbers $\omega_f^{\rm ev}\in\R$, 
$\omega_f^{\rm od}\in i\R$ such that the coefficients of $r_f^{\rm ev}(X)/\omega_f^{\rm od}$ and 
$r_f^{\rm od}(X)/\omega_f^{\rm ev}$ and the number $\omega_f^{\rm ev}\omega_f^{\rm od}/i\la f,f \ra$ 
belong to the number field $\mathbb{Q}_f$ generated by the Fourier coefficients of $f$, and transform 
by $\sigma$ if $f$ is replaced by $f^{\sigma} = \sum_{n\geq 1}\sigma \bigl(a_n(f)\bigr)\,q^n$ with
$\sigma \in {\rm Gal}(\overline{\Q}/\Q)$.  (See \cite{M1} and Chapter~V of~\cite{L}.) 
For instance,  for the unique newform $f=q-8q^2+12q^3+64q^4+\cdots$ in $S_8(\G_0(2))$ we have
$r_f^{\rm ev}(X)/\omega_f^{\rm od}=8X^6 - 34 X^4 + 17X^2 -1$, 
$r_f^{\rm od}(X)/\omega_f^{\rm ev}=4X^5 - 5X^3 + X$ and 
$\omega_f^{\rm ev}\omega_f^{\rm od}/i\la f,f \ra=32/17$  with $\omega_f^{\rm od}=0.001759\cdots i$
and $\omega_f^{\rm ev}=0.01049\cdots$.  It follows just as in the level~one case that the polynomials $R_f(X,Y)$
defined by~\eqref{Rfdef} have coefficients in~$\Q_f$ for $f\in\BBc_{k,N}$ and that the cuspidal part 
of $\CC_N(X,Y,\t,T)$ belongs to $\Q[X,Y][[q,T]]$. 

For the theorem, we also need to treat the case of non-cusp forms. Here neither the integral
defining $r_f(X)$ nor that defining $\la f,f\ra$ converges, but in~\cite{Z1991} extensions of
both quantities were defined, the main differences with the cuspidal case being that $r_f(X)$
no longer belongs to $\mathbb{V}_{k-2}$ but to the slightly bigger space
\bes \hV_{k-2} \= \la X^{-1},\,1,\, \dots, \,X^{k-1}\ra \, 
  \= \mathbb{V}_{k-2}\,\oplus\,\C\cdot X^{-1}\,\oplus\,\C\cdot X^{k-1}\,. \ees
and that $\la f,f\ra$ can be negative.  The definitions in both cases are simple: the
Petersson product $\la f,f\ra$ is defined as the same multiple of the residue at $s=k$
of $\sum a_f(n)^2\,n^{-s}$ as in the cuspidal case (it turns out that this $L$-series
has a simple pole at $s=k$ whether $f$ is cuspidal or not), and the period ``polynomial"
$r_f(X)$ is defined by the same formula~\eqref{Newrfdef} as in the cuspidal case, which
is again independent of the choice of~$\t\in\H$, but with $\widetilde f(X,\t)$ now defined by
\bes \widetilde f(X,\t) \,=\, \int_\t^\infty \bigl(f(\t')-a_0(f)\bigr)\,(X-\t')^{k-2}\,d\t'\+
      a_0(f)\,\frac{(X-\t)^{k-1}}{k-1} \ees
rather than by~\eqref{rtildedef}.  The coefficients of the Laurent polynomial $r_f(X)$ are related to
the special values of the $L$-series of~$f$ essentially as before. Since the $L$-series of $G_k$ is 
just $\zeta(s)\zeta(s-k+1)$, it is simple to use these definitions to calculate the contribution 
$C_k^{\rm Eis}(X,Y,\t)=\frac1{(k-2)!}R_{G_k}(X,Y)G_k(\t)$ of the Eisenstein series $G_k\in\BB_{k,1}$ 
to $C_k(X,Y,\t)$.  The result, given in~\cite{Z1991}, is 
  \be\label{CkEis1}
  C_k^{\rm Eis}(X,Y,\t) \= \bigl[(1-X^{k-2})Q_k(Y) + (1-Y^{k-2})Q_k(X)\bigr]\, E_k(\t)\,,\ee
where $E_k(\t)=G_k(\t)/G_k(\infty)=1-\frac{2k}{B_k}q-\cdots (B_k=$ the $k$th Bernoulli number) is the Eisenstein series normalized to have
the value~1 at~$\infty$,  where  
\be \label{defQk} Q_k(X) \= \sum_{\substack{r,\,s\ge0\text{ even}\\ r+s=k}}\frac{B_r}{r!}\,\frac{B_s}{s!}\,X^{r-1}
   \;\quad\in\quad \hV_{k-2}^{\rm odd} \;. \ee
(Notice that the condition ``$r$ and $s$ even" is not needed unless $k=2$, in which case the 
expression~\eqref{CkEis1} is zero anyway.)  From the definitions just given, it is clear that
both the results~\eqref{PSPRatio} and~\eqref{relPer} relating the Petersson norm and period polynomial
of an oldform $f=\L_{k,N_2}^{\e_2}(f_1)$ to the corresponding invariants of~$f_1$ remain true in the
non-cuspidal case, so taking $f_1=G_k$, $N_2=N$ and $\e_2=\e$ one immediately gets the corresponding
results for the Eisenstein case.  The result will be given in the next section, in which we compute the
total contribution of all of the Eisenstein series in $M_{*,N}$ to the generating function~$\CC_N$.

This completes our discussion of the definitions and main properties of all of the quantities appearing
in Theorem~\ref{Thm1}.

\section{Eisenstein series on $\G_0(N)$ and their periods}  \label{Eis}

Before proceeding, we introduce a notational convention that will be useful both for the proofs and for the
discussions of the numerical examples.  This is to decompose each
odd function of $X$ and~$Y$ as the sum of its ``even-odd" part and its ``odd-even part" (obtained by
interchanging $X$ and~$Y$), and to denote the first of these by the corresponding German (fraktur) letters.
Thus we write $R_f(X,Y)$ as $\Rg_f(X,Y)+\Rg_f(Y,X)$ with
  \be\label{Revod} \Rg_f(X,Y) \,=\, \frac{R_f(X,Y)+R_f(-X,Y)}2  
 \,=\, \frac{r_f^{\rm ev}(X)\,r_f^{\rm od}(Y)}{(2 i)^{k-3} \,\langle f,\, f\rangle}\ee
and similarly $C_k(X,Y,\t)=\Cg_k(X,Y,\t)+\Cg_k(Y,X,\t)$, etc., for the corresponding generating functions.  Then we have, for example,
\be \label{defevodgenfn} \Cg_N(X,Y,\t,T)\=\frac{NX^2-1}{NX^2YT^2}\+\sum_{k=2}^\infty \Cg_{k,N}(X,Y,\t)\,T^{k-2}\;. \ee

It is also convenient to introduce the notation $\BC_N(X,Y,\t,T)$ for the right-hand sides of~\eqref{mm}, so that
the statement of Theorem~\ref{Thm1} can be written simply as $\CC_N=\BC_N$.  The object of this section is to show that 
at least the Eisenstein parts of $\CC_N$ and~$\BC_N$ agree, i.e., that the difference between $\CC_N(X,Y,\t,T)$ and $\BC_N(X,Y,\t,T)$
vanishes in all of the cusps.  Since these cusps are all obtained from the cusp at infinity by applying 
Atkin-Lehner involutions, as discussed in~\S\ref{PerPolys}, it is enough for this to show that
 \be\label{CEisB}  (\CC_N|W_M)(X,Y,\infty,T)  \=  (\BC_N|W_M)(X,Y,\infty,T)  \qquad\forall\;M\in\Dg(N)\,,  \ee
where $|W_M$ is the operator given by applying $|_kW_M$ (with respect to the variable~$\t$) to the coefficient 
of~$T^{k-2}$ for each~$k\ge0$. To prove~\eqref{CEisB}, we will calculate both generating series independently and show their equality.  

To calculate the left-hand side of~\eqref{CEisB} we need only the contribution of the Eisenstein series.
As we have already discussed, the space $M_{k,N}^{\rm Eis}$ of Eisenstein series of weight~$k$ on $\G_0(N)$ 
has dimension $\nu(N)=2^t$ if \hbox{$k>2$} (and one less if~$k=2$).  It has three natural bases: the forms
$(G_k\circ d)(\t):=G_k(d\t)$ with $d$ ranging over~$\Dg(N)$, the forms $G_{k,N}^\e:=\L_{k,N}^\e(G_k)$
with $\e$ ranging over~$\Dg(N)^\vee$, and the forms $E_{k,N}^{(P)}$ with $P$ ranging over the cusps
of~$\G_0(N)$, where $E_{k,N}^{(P)}$ denotes the Eisenstein series that equals~1 (in a suitable sense)
at~$P$ and~0 at all of the other cusps of~$\G_0(N)$. Since the group of Atkin-Lehner involutions acts
simply transitively on the cusps of~$\G_0(N)$, the Eisenstein series $E_{k,N}^{(P)}$ for the cusp
$P=W_M(\infty)$ with $M|N$ (which is the cusp represented by any rational number whose denominator has
g.c.d.~$N/M$ with~$N$) can be taken simply to be~$E_{k,N}^{(\infty)}|_k^{\vphantom\e}W_M$, where
\be\label{EisInf} E_{k,N}^{(\infty)} \= \sum_{\g\,\in \G_0(N)_\infty \backslash\G_0(N)} 1\,\Bigl|_k\g
  \= \frac{1}{\prod_{p|N}(p^k-1)}\,\sum_{d|N}\mu\Bigl(\frac Nd\Bigr)\,d^k\,E_k\circ d \ee
with $\mu(d)\,=$ M\"obius function and $E_k(\t)=\frac{G_k(\t)}{G_k(\infty)}$ as before.  In this section
we will need only the basis of Hecke forms $G_{k,N}^\e$, but the proof of Theorem~\ref{Thm1} in~\S\ref{Proof}
will use all three bases: the Eisenstein series associated to the cusps of $\G_0(N)$ are the ones needed to
apply the Rankin-Selberg method, the Eisenstein series $G_k\circ d$ coming directly from level one are the ones 
that will appear in the expansion of~$\BC_N$ as a sum of Rankin-Cohen brackets, and the Hecke forms $G_{k,N}^\e$ 
are the convenient ones when we need to exploit the orthogonality with respect to the Petersson product of modular
forms on~$\G_0(N)$ with different Atkin-Lehner eigenvalues.

We begin by computing the Eisenstein part of $C_{k,N}(X,Y,\t)$, which we can write with the convention
introduced above as $\Cg^{\rm Eis}_{k,N}(X,Y,\t)+\Cg^{\rm Eis}_{k,N}(Y,X,\t)$, where
\be\label{CgkN} \Cg^{\rm Eis}_{k,N}(X,Y,\t) \= \frac1{(k-2)!}\,\sum_{\e\in\Dg(N)^\vee} \Rg_{G_{k,N}^\e}(X,Y)\,G_{k,N}^\e(\t)\,. \ee
(In principle, we should add the condition ``$\e\ne1$ if $k=2$" to the summation, but this is not necessary since the
symmetry property of periods implies that $r_f^{\rm ev}(X)=0$ and hence $R_f(X,Y)=0$ for all~$f\in M_{2,N}^\e$ 
\hbox{if~$\e(N)=1$.)}  We can compute each summand in this expression directly from the level~one formula~\eqref{CkEis1} 
together with equations~\eqref{PSPRatio} and~\eqref{relPer} applied to the case \hbox{$N_1=1$}, $N_2=N$, $f_1=G_k$. 
The first of these, together with the formula $a_p(G_k)=p^{k-1}+1$ for $p$~prime, gives
$$ \frac{\la G_{k,N}^\e,\,G_{k,N}^\e\ra}{\la G_k,G_k\ra} 
  \= 2^t\,\prod_{p|N}\bigl(1\+\e(p)p^{k/2}\bigr)\bigl(1\+\e(p)p^{1-k/2}\bigr)\,.$$
The second can be written symbolically as $r_{G_{k,N}^\e}=\L_{2-k,N}^\e(r_{G_k})$ with an obvious notation, and since
$$  \L_{2-k,N}^\e\bigl(1-X^{k-2}\bigr) 
  \= \bigl(1 \,-\, \e(N)\,N^{k/2-1}X^{k-2} \bigr)\,\cdot\,\prod_{p|N}\bigl(1+\e(p)p^{1-k/2}\bigr) $$
by an easy calculation (consisting of interchanging the divisors $d$ and $N/d$ in the sum defining the
coefficient of $X^{k-2}$), we obtain from~\eqref{CkEis1} the equation
$$ \frac{\Rg_{G_{k,N}^\e}(X,Y)}{(k-2)!}\,G_{k,N}^\e(\t) \=
   \frac{1-\e(N)N^{k/2-1}X^{k-2}}{2^t\,\prod_{p|N}(1+\e(p)p^{k/2})} \, \L_{2-k,N}^\e(Q_k(Y))
   \, \frac{G_{k,N}^\e(\t)}{G_k(\infty)}\,. $$
If we now substitute into this the value
$$ (G_{k,N}^\e|_kW_M)(\infty) \= (\e(M)G_{k,N}^\e)(\infty) \= \e(M)\,\sum_{d|N}\e(d)d^{k/2}\,G_k(\infty)\;,$$
then the denominator cancels and we obtain 
\begin{align*} & \frac{\Rg_{G_{k,N}^\e}(X,Y)}{(k-2)!}\,(G_{k,N}^\e|_k^{\vphantom\e}W_M)(\infty) 
    \=  2^{-t}\,\Bigl(\e(M)\,-\,\e(N/M)\,N^{k/2-1}X^{k-2} \Bigr)   \\
   &\qquad \quad \times\, \Biggl( \sum_{\substack{r,\,s\ge0\text{ even}\\ r+s=k}}
   \frac{B_r}{r!}\,\frac{B_s}{s!}\,\sum_{d|N} \e(d)\,d^{(r-s)/2}\,Y^{r-1}\Biggr)\;.   \end{align*}
We now insert this formula into~\eqref{CgkN} and sum over all $\e\in\Dg(N)^\vee$,
using the identity $2^{-t}\sum_{\e}\e(d)\e(d')=\delta_{d,d'}$ for~$d,\,d'\in\Dg(N)$, to find
\begin{align*} (\Cg_{k,N}|_kW_M)(X,Y,\infty) \= \sum_{\substack{r,\,s\ge0\text{ even}\\ r+s=k}}
   \frac{B_r}{r!}\,\frac{B_s}{s!} &  \Bigl[ M^{\frac{r-s}2}-N^{r-1}M^{\frac{s-r}2}X^{k-2}\Bigr]\,Y^{r-1}\,.\end{align*} 
Substituting this in turn into~\eqref{defevodgenfn} and using the standard generating function identity
$2\sum_{\text{$r\ge0$ even}} B_rt^{r-1}/r!= \coth(t/2)$, we obtain 
$$ 4\,(\Cg_N|W_M)(X,Y,\infty,T) \= \coth\tfrac{\sqrt MYT}2\,\coth\tfrac T{2\sqrt M}
  \,-\, \coth\tfrac{NXYT}{2\sqrt M}\,\coth\tfrac{\sqrt MXT}2\,. $$
The symmetrization of this in $X$ and $Y$ is a product of two differences of hyperbolic cotangents, and using 
the identity $\,\coth a-\coth b=\dfrac{\sinh(b-a)}{\sinh(a)\sinh(b)}$
we obtain finally the following result describing the values of $\CC_N$ at all cusps:
\begin{proposition}\label{CNinf} For any divisor $M$ of $N$, we have
\begin{align*} & \qquad\qquad (\CC_N|W_M)(X,Y,\infty,T) \;\= \\
& \frac{\sinh(\sqrt M(X+Y)T/2)\,\sinh((1-NXY)T/2\sqrt M)}
  {4\,\sinh(\sqrt MXT/2)\,\sinh(\sqrt MYT/2)\,\sinh(T/2\sqrt M)\,\sinh(NXYT/2\sqrt M)} \;. \end{align*} 
\end{proposition} 

\null \bigskip

To complete the proof of $\CC_N^{\rm Eis}=\BC_N^{\rm Eis}$, we also have to compute the value of $\BC_N$ at
all cusps.  But this is much easier. We can write the definition of $\BC_N(X,Y,\t,T)$, the right-hand side
of~\eqref{mm}, as
\be \label{BN}  \BC_N(X,Y,\t,T) \= F_\t(XT,YT)\,F_{N\t}(T,-NXYT)\,, \ee
where
\be\label{Fth} F_\t(u,\,v) \= \frac{\theta'_\t(0)\,\theta_\t(u\+v)}{\theta_\t(u)\,\theta_\t(v)} \ee
with $\theta_\t$ as in~\eqref{thetadef}.  The function~$F_\t$ was defined and studied in \cite{Z1991} 
(but in fact already by Kronecker, as the author learned later) and will be used again in the next section. 
Here we need only its modular transformation property
\be\label{Fmod} F_{\frac{a\t+b}{c\t+d}}\Bigl(\frac u{c\t+d},\,\frac v{c\t+d}\Bigr)
 \= (c\t+d)\,\exp\Bigl(\frac{cuv/2\pi i}{c\t+d}\Bigr)\,F_{\t}(u,v) \ee
for $\left(\begin{smallmatrix}a & b \\ c & d\end{smallmatrix}\right) \in SL_2(\Z)$, which follows immediately
from the definition and from the standard transformation property of $\theta_\t$ with respect to
the modular group. From the definition of $\,|W_M\,$ we have
\be\label{BW} (\BC_N|W_M)(X,Y,\t,T) \= \frac M{(c\t+d)^2}\,
  \BC_N\biggl(X,\,Y,\,\frac{a\t+b}{c\t+d},\,\frac{\sqrt MT}{c\t+d}\biggr)\,,\ee
where　$\bigl(\begin{smallmatrix}a & b \\ c & d\end{smallmatrix}\bigr)$ is a matrix representing~$W_M$,
i.e., an integer matrix satisfying $M|a$, $N|c$, $M|d$ and $ad-bc=M$.  We have
$W_M = \alpha_M\j_M$ and $\j_NW_M = M\alpha_M^*\j_{N/M}$, where 
$V_n=\bigl(\begin{smallmatrix}n&0\\0&1\end{smallmatrix}\bigr)$ as in~\S\ref{PerPolys} and where
$\alpha_M=\bigl(\begin{smallmatrix}a/M & b \\ c/M & d\end{smallmatrix}\bigr) $
and $\alpha^*_M=\bigl(\begin{smallmatrix}a & Nb/M \\ c/N & d/M\end{smallmatrix}\bigr)$ both belong to $SL_2(\Z)$. 
Inserting~\eqref{BN} into~\eqref{BW} and using~\eqref{Fmod} for $\alpha_M$ and $\alpha^*_M$ we obtain
 $$ (\BC_N|W_M)(X,Y,\t,T)  \,=\, F_{M\t}\bigl(\sqrt MXT,\,\sqrt MYT\bigr)\,
 F_{\frac NM\t}\Bigl(\frac T{\sqrt M},\,-\frac{NXYT}{\sqrt M}\Bigr)$$
after a short calculation. (The two exponential terms coming from~\eqref{Fmod} cancel.)
Now letting $\t\to\infty$ and observing that the limiting value at $\t=\infty$ of $\theta_\t(z)/\theta'_\t(0)$
equals $\,2\sinh(z/2)$, we find the same limiting value as the one given in Proposition~\ref{CNinf}.
This completes the proof of equation~\eqref{CEisB}.

The calculation just given also lets us refine Theorem~\ref{Thm1} to a formula for each eigencomponent of $\CC_N$ 
under the action of the group of Atkin-Lehner involutions, which is useful both theoretically and for computational 
purposes (as will be illustrated in Section~\ref{Examples}).  Indeed, once we have finished the proof of Theorem~\ref{Thm1} and 
shown that $\CC_N=\BC_N$, we can write the above formula for $(\BC_N|W_M)(X,Y,\t,T)$ as a formula for $(\CC_N|W_M)(X,Y,\t,T)$, 
and the average of these expressions over all divisors~$M$ of~$N$, weighted with the value of~$\e(M)$ for 
some~$\e\in\Dg(N)^\vee$, gives the $\e$-eigencomponent of~$\CC_N$.  We state the result in the following theorem, 
which will be proved as soon as Theorem~\ref{Thm1} is:
\begin{theorem}\label{Thm3} For each even integer~$k\ge2$ and for each homomorphism
$\e\in\Dg(N)^\vee$ define $C_{k,N}^\e$ by the same formula as in~\eqref{Main1} but with the sum
restricted to $f\in\BB_{k,N}^\e$.  Then the four-variable generating function 
$$ C_N^\e(X,Y,\t,T) \df 
 \delta_{\e,1}\,\frac{(X+Y)(NXY-1)}{NX^2Y^2T^2}\+\sum_{k=2}^\infty C_{k,N}^\e(X,Y,\t)\,T^{k-2}$$
can be evaluated in terms of the theta function~\eqref{thetadef} as
$$ \frac1{2^t}\sum_{M|N}\e(M)\,
 \frac{\theta_{M\t}'(0)\,\theta_{M\t}\bigl(\sqrt M(X+Y)T\bigr)}{\theta_{M\t}(\sqrt MXT)\,\theta_{M\t}(\sqrt MYT)} \cdot
 \frac{\theta_{\frac NM\t}'\bigl(0)\,\theta_{\frac NM\t}\bigl(\frac{(NXY-1)T}{\sqrt M}\bigr)}
   {\theta_{\frac NM\t}\bigl(\frac{NXYT}{\sqrt M}\bigr)\,\theta_{\frac NM\t}\bigl(\frac T{\sqrt M}\bigr)\,}\,. $$
\end{theorem}

\section{Rankin-Cohen brackets and periods of cusp forms of level~1}  \label{Level1}

Write $\BC_{k,N}(X,Y,\t)$ for the coefficient of~$T^{k-2}$ in $\BC_N(X,Y,\t,T)$.  Then Theorem~\ref{Thm1} says that 
$\BC_{k,N}=\CC_{k,N}$ for all~$k$. Since both $\BC_{k,N}$ and $\CC_{k,N}$ are the sum of their Eisenstein and their
cuspidal parts, and since we have just proved the equality of the Eisenstein parts agree, it suffices to prove that 
the cuspidal parts also agree.  This in turn is equivalent to showing that $\la f,\BC_{k,N}\ra=\la f,\CC_{k,N}\ra$
for every~$f$ in the basis $\BB_{k,N}^{\rm cusp}$ of $S_{k,N}$, and in view of the definition~\eqref{Main1}
of $\CC_{k,N}$ and the orthogonality of the elements of~$\BB_{k,N}^{\rm cusp}$ this is equivalent to the formula
\be\label{CcuspB}  \bigl\la f,\, \BC_{k,N}(X,Y,\,\cdot\,)\bigr\ra \= 
 \frac{r_f^{\rm ev}(X)\,r_f^{\rm od}(Y) \+ r_f^{\rm od}(X)\,r_f^{\rm ev}(Y)} {(2i)^{k-3}(k-2)!}\,.  \ee
Our proof of this equality will be modelled on the proof given in~\cite{Z1991} for the level~1 case.
The main ingredients of that proof were the Rankin-Cohen brackets of two modular forms and
their modifications when one or both of the arguments is replaced by the quasimodular Eisenstein
series~$G_2$, an identity of Rankin and Zagier expressing the Petersson product of a Hecke cusp form~$f$ with such 
a bracket as a product of periods of~$f$, and the Laurent expansion of the function $F_\t(u,v)$ defined in~\eqref{Fth}.
In this section we review each of these things and present the proof of~\eqref{CcuspB} for $N=1$ given in~\cite{Z1991}
in a form that is a little simpler than the one there and that makes the generalization to the case of arbitrary
squarefree level as simple as possible.  This generalization will then be carried out in the following section.

The {\it Rankin-Cohen bracket} of two holomorphic functions $F$ and $G$ in~$\H$ is the bilinear 
combination of derivatives
 $$ [F,\,G]_m^{(k_1,k_2)} \,=\! \sum_{\substack{m_1,\,m_2\ge0\\ m_1+m_2=m} }
  (-1)^{m_2}\binom{k_1+m-1}{m_2}\binom{k_2+m-1}{m_1}D^{m_1}(F)\,D^{m_2}(G), $$
where $k_1,\,k_2>0$, $m\ge0$ are integers and $D=D_\t=\frac1{2\pi i}\,\frac d{d\tau}=q\,\frac d{dq}\,$.
This definition was found by Cohen~\cite{Cohen}, who
proved that $[F|_{k_1}\,g,G|_{k_2\,}g]_m^{(k_1,k_2)} =[F,G]_m^{(k_1,k_2)}|_{k_1+k_2+2m}\,g$ for any 
\hbox{$g\in GL_2^+(\R)$}. In particular, if $F$ and~$G$ are modular of weights~$k_1$ and~$k_2$ on
some Fuchsian group, then $[F,G]_m^{(k_1,k_2)}$ is modular of weight $k=k_1+k_2+2m$ on the same group.
In the case when $F$ and~$G$ are Eisenstein series of level~1, one has the {\it modified  
Rankin-Cohen bracket}~\cite{Z1991} 
\be\label{MRCB}\begin{aligned} \bigl[G_{k_1},\,G_{k_2}\bigr]_m & \df \bigl[G_{k_1},\,G_{k_2}\bigr]_m^{(k_1,k_2)} \\
 & \qquad \+ \frac{\delta_{k_2,2}}2\,\frac{D^{m+1}(G_{k_1})}{m+k_1} 
 \+ \frac{(-1)^m\delta_{k_1,2}}2\,\frac{D^{m+1}(G_{k_2})}{m+k_2}\,, \end{aligned} \ee
which is still modular of the same weight~$k$ even if one or both of $k_1$ or~$k_2$  equals~2, in 
which case the corresponding Eisenstein series $G_{k_1}$ or~$G_{k_2}$ is only quasimodular.
If $f\in S_k$ is a Hecke cusp form of weight~$k$, then results of Rankin~\cite{R} and 
Zagier~\cite{Z1977,Z1991} say that the Petersson scalar product of~$f$ with
this modified Rankin-Cohen bracket is given in all cases by
\be\label{RZ} -\,(2i)^{k-1}\,\bigl\la f,\,[G_{k_1},G_{k_2}]_m\bigr\ra\= \binom{k-2}m\,r_m(f)\,r_{m+k_1-1}(f)\,.\ee
This formula, whose proof is based on the Rankin-Selberg convolution method, will be generalized to our case 
in the next section (Proposition~\ref{MRCBN}).

The other ingredient of the proof of~\eqref{Zmain3} in~\cite{Z1991} was a formula expressing the Laurent
coefficients at the origin of the meromorphic function $F_\t$ defined by~\eqref{Fth} as derivatives of 
Eisenstein series, namely
\be\label{Laur} F_\t(u,v) \= \sum_{k>0,\,m\ge-1} g_{k,m}(\t)\,\bigl(u^{k-1}\+v^{k-1}\bigr)\,(uv)^m \ee
with $g_{k,m}$ (which is non-zero only for $k$ even) defined by
\be\label{gkm}  g_{k,m}(\tau) = \begin{cases} \dfrac{-2\,D^m G_k(\t)}{m!(m+k-1)!} 
  &\text{if $m\ge0$}, \\  \qquad \delta_{k,2} & \text{if $m=-1$}. \end{cases}\ee

With these preparations, the proof of the identity~\eqref{CcuspB} in the case~$N=1$ is quite easy.  The basic 
observation is that the modified Rankin-Cohen bracket of Eisenstein series defined above, rescaled by a convenient
factor, can be written uniformly in all cases as
$$g_{k_1,k_2,m} \,:=\, \frac{4\,[G_{k_1},\,G_{k_2}]_m}{(k_1+m-1)!\,(k_2+m-1)!} 
  \,= \sum_{\substack{m_1,\,m_2\ge-1\\ m+m_2=m}} (-1)^{m_2}g_{k_1,m_1}\,g_{k_2,m_2}\,,$$
where the terms $m_1=-1$ and $m_2=-1$ correspond to the correction terms needed in the definition~\eqref{MRCB}
when $k_1$ or $k_2$ equals~2.  Inserting~\eqref{Laur} 
into~\eqref{BN} and comparing the coefficients of~$T^{k-2}$ on both sides, we therefore find
$$ B_{k,1}(X,Y,\t)  \,=\, \sum_{\substack{k_1,\,k_2>0,\;m\ge0\\ k_1+k_2+2m=k} } 
  (X^{k_1-1}+Y^{k_1-1})(1-(XY)^{k_2-1})(XY)^m\,g_{k_1,k_2,m}(\t)\;. $$
From this and~\eqref{RZ} we find that the scalar product of a Hecke form~$f\in\BBc_k$ with~$B_{k,1}$ is given by
\begin{align*} & (2i)^{k-3}(k-2)!\,\bigl\la f,\, B_{k,1}(X,Y,\,\cdot\,)\big\ra 
  \=  \sum_{\substack{k_1,\,k_2>0,\;m\ge0\\ k_1+k_2+2m=k} } \binom{k-2}m\,\binom{k-2}{m+k_1-1} \\
  & \qquad\qquad \cdot \,r_m(f)\,r_{m+k_1-1}(f)\,(X^{k_1-1}+Y^{k_1-1})(1-(XY)^{k_2-1})(XY)^m  \\
&\qquad\= \sum_{\substack{0\le i,\,j\le k-2\\ i\not\equiv j\;({\rm mod}\;2)} } \binom{k-2}i\,\binom{k-2}j\,r_i(f)\,r_j(f)\,X^i\,Y^j\\
&\qquad\= \; r_f^{\rm ev}(X)\,r_f^{\rm od}(Y)\+r_f^{\rm od}(X)\,r_f^{\rm ev}(Y)\,.
\end{align*}
(Here the second equality follows by breaking up the set of pairs $(i,j)$ with $i\not\equiv j\!\pmod 2$
into four subsets according as $i\gtrless j$ and $i+j\gtrless k-2$ and using the symmetry property
$r_{k-2-i}(f)r_{k-2-j}(f)=-r_i(f)r_j(f)$.)  This completes the proof of~\eqref{CEisB} for $N=1$ and hence
of the equality $\BC_1=\CC_1\,$.

\section{Proof of Theorem 1}  \label{Proof}

To carry out the corresponding proof in the case of squarefree level~$N$, we must define the extended Rankin-Cohen
brackets for all pairs of Eisenstein series of the same weight~$k$ and compute their scalar products with both
old and new Hecke forms $f$ in $S_{k,N}$ in terms of the periods of~$f$. We will need to work with all three
bases $\{G_k\circ d\mid d\in\Dg(N)\}$, $\{G_{k,N}^\e\mid\e\in\Dg(N)^\vee\}$ and 
$\{G_{k,N}^{(\infty)}|_kW_M\mid M\in\Dg(N)\}$ discussed in~\S\ref{Eis}.  As we already said there, each of these 
will be the best choice for some part of our calculation.  
In particular, the modified Rankin-Cohen bracket that we will need
is $[G_{k_1},G_{k_2}\circ N]_m$, which we define by exactly the same formula as in~\eqref{MRCB}
but with $G_{k_2}$ replaced by $G_{k_2}\circ N$. The basic statement that we need is the formula for 
its scalar product with Hecke cusp forms given by the following proposition.
\begin{proposition} \label{MRCBN} For $k_1,k_2>0$ even and $m\ge0$ the function
\begin{align}\label{defgNk1k2m} g^{(N)}_{k_1,k_2,m} 
  & \df \frac{4\,N^{-k_2/2}}{(k_1+m-1)!\,(k_2+m-1)!} \bigl[G_{k_1},\,G_{k_2}\circ N\bigr]_m \end{align}
is a modular form of weight $k=k_1+k_2+2m$ on~$\G_0(N)$, and its Petersson scalar product with any
Hecke cusp form $f\in\BBc_{k,N}$ is given by
\be\label{RZN} \la f,g^{(N)}_{k_1,k_2,m}\ra 
  \= \binom{k-2}m\,\binom{k-2}{m+k_1-1}\,\frac{r_m(f)\,r_{m+k_1-1}(f)}{(2i)^{k-3}(k-2)!}\;.  \ee
\end{proposition}

Given this proposition, the proof of \eqref{CcuspB} follows exactly as in the level~1 case.  
The definition of $g^{(N)}_{k_1,k_2,m}$ can be rewritten in the form
$$ g^{(N)}_{k_1,k_2,m}(\t)\=\sum_{\substack{m_1,\,m_2\ge-1\\ m+m_2=m}}
  (-N)^{m_2} g_{k_1,m_1}(\t)\,g_{k_2,m_2}(N\t)\,,$$
so from~\eqref{Laur} we see just as before that the coefficient $B_{k,N}(X,Y,\t)$ of $T^{k-2}$ 
in~$\BC_N(X,Y,\t,T)$ has the expansion
$$ \sum_{\substack{k_1,\,k_2>0,\;m\ge0\\ k_1+k_2+2m=k} } 
  (X^{k_1-1}+Y^{k_1-1})(1-(NXY)^{k_2-1})(XY)^m\,g^{(N)}_{k_1,k_2,m}(\t)\;. $$
This already proves the first statement of the proposition above (which can, of course,
be established in several other ways), since the modularity properties of $\theta_\t$ or $F_\t$ imply that 
$B_{k,N}(X,Y,\t)$ is modular of weight~$k$ and level~$N$ in~$\t$. In combination with~\eqref{RZN} it gives
\begin{align*} & (2i)^{k-3}(k-2)!\,\bigl\la f,\, B_{k,N}(X,Y,\,\cdot\,)\big\ra 
  \=  \sum_{\substack{k_1,\,k_2>0,\;m\ge0\\ k_1+k_2+2m=k} } \binom{k-2}m\,\binom{k-2}{m+k_1-1} \\
  & \qquad\qquad \cdot \,r_m(f)\,r_{m+k_1-1}(f)\,(X^{k_1-1}+Y^{k_1-1})(1-(NXY)^{k_2-1})(XY)^m  \\
&\qquad\= \sum_{\substack{0\le i,\,j \\ i\not\equiv j\;({\rm mod}\;2)} } \binom{k-2}i\,\binom{k-2}j\,r_i(f)\,r_j(f)\,X^i\,Y^j\\
&\qquad\=\; r_f^{\rm ev}(X)\,r_f^{\rm od}(Y)\+r_f^{\rm od}(X)\,r_f^{\rm ev}(Y)
\end{align*}
for any $f\in\BBc_{k,N}$ by exactly the same computation as before except that now the symmetry property
of the periods used is $r_{k-2-m}(f)r_{m+k_2-1}(f)=-N^{k_2-1}r_m(f)r_{m+k_1-1}(f)$.  This completes the 
proof of Theorem~\ref{Thm1} assuming equation~\eqref{RZN}, so it remains only to prove this equation.
 
The main tool needed for this proof is the Rankin-Selberg method.  In its simplest form this
gives the formula $\la f,gE_{k_2}\ra\,=\,\frac{(k-2)!}{(4\pi)^{k-1}}\,L(f*g,k-1)$ for any $f\in S_k$ and 
any $g\in M_{k_1}$ with real Fourier coefficients, where $k=k_1+k_2$ and $L(f*g,s)$ denotes the convolution 
$L$-series $\,\sum_{n=1}^\infty a_n(f)\,a_n(g)\,n^{-s}\,$ (or its meromorphic continuation). 
More generally, it was shown in~\cite{Z1977} that
\be\label{RS} \langle f, \bigl[g, E^{(\infty)}_{k_2,N} \bigr]_m \rangle 
  \=  \frac{(k-2)!\,(k_2+m-1)!}{(4\pi)^{k-1}(k_2-1)!}\,L(f*g,\,k-m-1) \,, \ee
for any~$N$, any $g\in M_{k_1,N}$ with real coefficients, any $f\in S_{k,N}$ and any $m\ge0$, where now $k=k_1+k_2+2m$.
If $N=1$, the cusp form $f$ is a Hecke eigenform, and the function $g$ is the Eisenstein series $G_{k_1}$, then 
a well-known elementary calculation (which we will repeat below) shows that 
$L(f*g,s)\,=\,\frac{L(f,s)L(f,s-k_1+1)}{\zeta(2s-k-k_1+2)}$, and this in conjunction with the relation between
the periods of~$f$ and the special values of its \hbox{$L$-function} at arguments $s\in\{1,\dots,k-1\}$ gives the
equation~\eqref{RZ} used for the proof of Theorem~\ref{Thm1} in the level~1 case.  For the case of
squarefree level~$N$, a similar formula holds in principle for any Hecke form $f\in\BB_{k,N}^{\rm cusp}$
and any $g\in M_{k_1,N}^{\rm Eis}$, but with two new aspects: first of all, $f$ can be an old- or newform
and we must treat both cases, and secondly, the formula for the convolution of $f$ and~$g$ now has the form
\be\label{factorA}  L(f*g,s) \= A(s)\,\frac{L(f,s)L(f,s-k_1+1)}{\zeta(2s-k-k_1+2)}\,,\ee
where $A(s)$ is a rational function of all $p^s$ with prime $p$ dividing~$N$ that 
depends on the particular Eisenstein series~$g$ chosen. Also, if we compare~\eqref{RS} with the 
equation~\eqref{RZN} that we want to prove, then we see that the Rankin-Cohen bracket occurring
in~\eqref{defgNk1k2m} does not have the form $[g,E_{k_2,N}^{(\infty)}]_m$ that we need. 
To complete the proof, we therefore must find a specific choice of the Eisenstein series~$g$ such that on the one
hand the Petersson scalar products of the cusp form~$f$ with $[G_{k_1},G_{k_2}\circ N]_m$ and with
$[g,G_{k_2,N}^{(\infty)}]_m$ are proportional, and on the other hand the value of the factor~$A(s)$ in~\eqref{factorA}
at $s=k-m-1$ is equal to~1.  The lemma below says that such an Eisenstein series exists and can be chosen independent of the 
integer~$m$ and of the Hecke form~$f$, but does depend on the eigenvalue~$\e$ of~$f$ with respect to the 
group of Atkin-Lehner involutions (as well as of course on~$N$, $k_1$ and~$k_2$).  

\begin{lemma} For fixed even integers $k_1,k_2>0$ and for every $\e\in\Dg(N)^\vee$, 
define an Eisenstein series $G_{k_1,N}^{\e,k_2}$ in $M_{k_1,N}^{\rm Eis}$ by
\be\label{defGk}  G_{k_1,N}^{\e,k_2} \df \sum_{d|N}\e(d)\,d^{\frac{k_1-k_2}2}\,G_{k_1}\circ d \,. \ee
Then for $m\ge0$, $k=k_1+k_2+2m$ and any $f\in S_{k,N}^\e$, we have
\be\label{GkProp} \bigl\la f,\,\bigl[G_{k_1},G_{k_2}\circ N\bigr]_m\bigr\ra 
  \= G_{k_2}(\infty)\,\cdot\,\bigl\la f,\,[G_{k_1,N}^{\e,k_2},\,E_{k_2,N}^{(\infty)}]_{\vphantom{x^{x^y}}m}\big\ra \ee
and
\be \label{conv} L\bigl(f*G_{k_1,N}^{\e,k_2},\,k-m-1\bigr) \= \frac{L(f,k-m-1)L(f,k_2+m)}{\zeta(k_2)}\,. \ee
\end{lemma}
\begin{proof} We write $G^\e$ for $G_{k_1,N}^{\e,k_2}$ for simplicity. Because of the orthogonality of the different 
eigenspaces of the group of Atkin-Lehner involutions, to prove \eqref{GkProp} it suffices to show that
\be\label{orthog} \bigl(\bigl[G_{k_1},\,G_{k_2}\circ N\bigr]_m\bigr)^{(\e)} \=
  G_{k_2}(\infty)\,\cdot\,\bigl(\bigl[G^\e,\,E_{k_2,N}^{(\infty)}\bigl]_{\vphantom{x^{x^y}}m}\bigr)^{(\e)}  \ee
for any $m\ge0$, where $k=k_1+k_2+2m$ as before and where we denote by $F^{(\e)}$ the $\e$-eigencomponent 
of a modular form~$F$ on~$\G_0(N)$.
The proof of this is an algebraic juggling game using the different
bases of $M_{*,N}^{\rm Eis}$.  First of all, applying the orthogonality relation $\sum_\e\e(d)\e(d')=2^t\delta_{d,d'}$
to the definition $G_{k,N}^\e=\sum_{d|N}\e(d)\,d^{k/2}G_k\circ d$, we get the expression
\be\label{GetoGd}  G_k\circ d \= \frac1{2^t\,d^{k/2}}\sum_{\e\in\Dg(N)^\vee} \e(d)\,G_{k,N}^\e \qquad(d\in\Dg(N))\ee
for the eigenfunction decomposition of each $G_k\circ d$.  Applying this with $(k,d)$ replaced by $(k_1,1)$
and by~$(k_2,N)$, and observing that the Rankin-Cohen bracket of an $\e_1$-eigenfunction and an $\e_2$-eigenfunction
is an $\e_1\e_2$-eigenfunction (because of the basic modular equivariance property of the bracket), we obtain
\bes  \bigl(\bigl[G_{k_1},\,G_{k_2}\circ N\bigr]_m\bigr)^{(\e)} \=\frac{4^{-t}}{N^{k_2/2}}
 \sum_{\substack{\e_1,\,\e_2\in\Dg(N)^\vee\\ \e_1\e_2=\e}}\e_2(N)\,\bigl[G_{k_1,N}^{\e_1},\,G_{k_2,N}^{\e_2}\bigr]_m\;.\ees
On the other hand, from equations~\eqref{defGk} and~\eqref{GetoGd} we obtain
\begin{align*} G^\e  &\= \frac1{2^t}\,\sum_{d|N}\e(d)\,d^{-k_2/2} \sum_{\e_1\in\Dg(N)^\vee}\e_1(d)\,G_{k_1,N}^{\e_1} \\
  &\= \frac1{2^t}\,\sum_{\e_1\in\Dg(N)^\vee}\; \prod_{p|N}\Bigl(1+\e\e_1(p)p^{-k_2/2}\Bigr)\;G_{k_1,N}^{\e_1}   \end{align*}
and from equations~\eqref{EisInf} and~\eqref{GetoGd} we obtain
\begin{align*} G_{k_2}(\infty)\; E_{k_2,N}^{(\infty)}  &\= \frac1{2^t\prod_{p|N}(p^{k_2}-1)}\;\sum_{d|N}\mu\Bigl(\frac Nd\Bigr)\,d^{k_2/2}\,
   \Biggl(\sum_{\e_2\in\Dg(N)^\vee}\e_2(d)\,G_{k_2,N}^{\e_2}\Biggr) \\
  &\= \frac1{2^t} \sum_{\e_2\in\Dg(N)^\vee}\;\prod_{p|N}\Bigl(\frac1{\e_2(p)p^{k_2/2}+1}\Bigr)\;G_{k_2,N}^{\e_2}\,,  \end{align*}
and combining these two equations gives
\bes G_{k_2}(\infty)\, \bigl(\bigl[G^\e,\,E_{k_2,N}^{\infty}\bigr]_m\bigr)^{(\e)} 
 \,=\, \frac{4^{-t}}{N^{k_2/2}} \sum_{\substack{\e_1,\,\e_2\in\Dg(N)^\vee\\ \e_1\e_2=\e}}
  \e_2(N)\,\bigl[G_{k_1,N}^{\e_1},\,G_{k_2,N}^{\e_2}\bigr]_m\;.  \ees
Comparing this with the previous result we obtain equation~\eqref{orthog}. 

To establish~\eqref{conv}, we will show first that~\eqref{factorA} holds for $g=G_{k_1,N}^{\e,k_2}$ for some function~$A(s)$
that is a product over all prime factors of~$N$ of functions $A_p(p^{-s})$ each of which has the value~1 at $s=k-m-1$.  This works because 
the $L$-series of both $f$ and $G^\e$ have Euler products, and we can therefore work one prime at a time.
Explicitly, the $L$-series of $G^\e$ is given by
\begin{align*} L(G^\e,s) & \= \sum_{d|N}\e(d)\,d^{\frac{k_1-k_2}2-s}\,\zeta(s)\,\zeta(s-k_1+1) \\
 & \= \prod_{p|N} \bigl(1\+\e(p)p^{\frac{k_1-k_2}2-s}\bigr)\,\cdot
      \prod_{\text{all $p$}}\frac1{(1-p^{-s})\,(1-p^{k_1-1-s})} \,. \end{align*}
We write any $L$-series $L(s)$ with an Euler product as $\prod_pL_p(p^{-s})$ where $L_p(X)$ is a rational function of~$X$.  
By the discussion in Section~\ref{PerPolys}, we can write any Hecke form $f\in\BB_{k,N}^\e$ as $f=\L_{k,N_2}^{\e_2}(f_1)$
with $f_1\in\BB_{k,N_1}^{\rm new,\e_1}$ for some decomposition $N=N_1N_2$ and corresponding decomposition $\e=\e_1\e_2$, 
and then $L(f,X)_p$ is given by equation~\eqref{oldLfactor} in terms of~$L(f_1,X)_p$, which in turn is given by 
equation~\eqref{newLfactor} with $N$, $\e$ and $f$ replaced by $N_1$, $\e_1$ and~$f_1$.  The local calculation 
splits into three cases, according as $p\nmid N$, $p|N_1$ or $p|N_2$. 

\bigskip \noindent {\it Case 1}\,: $p\nmid N$.  In this case we have 
$$ L(G^\e,X)_p  \= \frac1{(1-X)(1-p^{k_1-1}X)} \= \sum_{i=0}^\infty \frac{p^{(k_1-1)(i+1)}-1}{p^{k_1-1}-1}\,X^i $$
and
$$ L(f,X)_p  \= \frac1{1-a_p(f)X+p^{k-1}X^2} \= \sum_{i=0}^\infty \frac{\alpha^{i+1}-\beta^{i+1}}{\alpha-\beta}\,X^i $$
where $\alpha+\beta=a_p(f)$, $\alpha\beta=p^{k-1}$.  It follows that
\begin{align*}  & L(f*G^\e,X)_p \= \sum_{i=0}^\infty \frac{p^{(k_1-1)(i+1)}-1}{p^{k_1-1}-1}\,
    \frac{\alpha^{i+1}-\beta^{i+1}}{\alpha-\beta}\,X^i \\  
  &\quad \= \frac1{(\alpha-\beta)(p^{k_1-1}-1)}\,\biggl[\frac{\alpha p^{k_1-1}}{1-\alpha p^{k_1-1}X} - 
    \frac{\beta p^{k_1-1}}{1-\beta p^{k_1-1}X} -\frac{\alpha}{1-\alpha X} +\frac{\beta}{1-\beta X}\biggr] \\
  &\quad \= \frac{1 \,-\, p^{k+k_1-2}\,X^2}{(1-\alpha p^{k_1-1}X)(1-\beta p^{k_1-1}X)(1-\alpha X)(1-\beta X)} \\
  & \quad \= \frac{L(f,X)_p\,L(f,p^{k_1-1}X)_p}{\zeta(p^{k+k_1-2}X^2)_p}\,,\end{align*}
establishing the prime-to-$N$ part of equation~\eqref{orthog}.  This is the standard method of calculation, but in fact
there is a simpler way that does not require factoring the denominator of $L(f,X)_p$ into linear factors: one 
decomposes $L(G^\e,X)_p$ by partial fractions in the form $\frac{c_1}{1-X}+\frac{c_2}{1-p^{k_1-1}X}$
and then gets $L(f*G^\e,X)_p$ as $c_1L(f,X)_p+c_2L(f,p^{k_1-1}X)_p$.

\bigskip \noindent {\it Case 2}\,: $p|N_1$.  This is the ``$p$-new" case, since $p\nmid N_2$ and therefore $f$ is a newform as 
far as the prime~$p$ is concerned. Equations \eqref{oldLfactor} and~\eqref{newLfactor} give
 $$ L(f,X)_p \= L(f_1,X)_p \= \frac1{1\+\e_1(p)\,p^{k/2-1}X}\= \sum_{i=0}^\infty (-p^{k/2-1}\e_1(p))^i\,\,X^i\, $$
while the formula for $L(G^\e,s)$ given above (here with $\e(p)=\e_1(p)$) gives
\be \label{GeXp}  L(G^\e,X)_p \= \frac{1 + \e_1(p)\,p^{(k_1-k_2)/2}\,X}{(1-X)(1-p^{k_1-1}X)}\,.  \ee
The calculation here is easier than in the generic case $p\nmid N$, because the form of $L(f,X)_p$ as a geometric
series means that its convolution with any power series in~$X$ is obtained simply by replacing $X$ by $-\e_1(p)p^{k/2-1}X$ 
in that power series.  We therefore get after a short calculation
\be\label{LandA} L(f*G^\e,X)_p \= \frac{1-p^{k_1+m-1}X}{1-p^{k+k_1-2}X^2} \, 
  \frac{L(f,X)_p\,L(f,p^{k_1-1}X)_p}{\zeta(p^{k+k_1-2}X^2)_p} \ee
and since the first factor takes the value 1 at $X=p^{-k+m+1}$, we have established the $p$-part of~\eqref{conv} also in this case. 

\bigskip \noindent {\it Case 3}\,: $p|N_2$.  This is the ``$p$-old" case, since $f$ comes from a form whose level does
not contain~$p$.  Here equations \eqref{oldLfactor} and~\eqref{newLfactor} give
 $$ L(f,X)_p \= \frac{1\+\e_2(p)\,p^{k/2}X}{1-a_p(f_1)X+p^{k-1}X^2} \,, $$
while $L(G^\e,X)_p$ is given by~\eqref{GeXp}, but with $\e_1(p)$ replaced by~$\e_2(p)$. As in Case 1 we can write each of
these Euler factors as a linear combination of two geometric series, obtaining for their convolution a sum of
four geometric series that can be evaluated by an elementary, though quite tedious, computation, or alternatively
write $L(G^\e,X)_p$ in the form $\frac{c_1}{1-X}+\frac{c_2}{1-p^{k_1-1}X}$ and get $L(f*G^\e,X)_p$ as 
$c_1L(f,X)_p+c_2L(f,p^{k_1-1}X)_p$.  The result of the computation can be written in the form
\begin{align*} & L(f*G^\e,X)_p  \= (1-p^{k+k_1-2}X^2)\,L(f,X)_p\,L(f,p^{k_1-1}X)_p \\ 
 & \qquad \+ \e_2(p)p^{\frac{k_1-k_2}2}X\,(1-p^{k-m-1}X)\,Q(X)
  \,L(f_1,X)_p\,L(f_1,p^{k_1-1}X)_p \;, \end{align*}
where $Q(X)=a_p(f_1)+\e_2(p)p^{\frac k2}-(p^{k_1-1}+1)p^{k-1}X - \e_2(p)p^{k_1+\frac{3k}2-2}X^2$ is an irrelevant
quadratic polynomial.  Since the second term vanishes at $X=p^{-k+m+1}$, this establishes the $p$-part of~\eqref{conv}
in this case as well and completes the proof of the lemma.
\end{proof}

Now combining equations~\eqref{defgNk1k2m}, \eqref{GkProp}, \eqref{RS} with $g=G^\e$, \eqref{conv}, \eqref{rnf} and
the formula $G_{k_2}(\infty)=(k_2-1)!\zeta(k_2)/(2\pi i)^{k_2}$ gives equation~\eqref{RZN},
completing the proof of Proposition~\ref{MRCBN} and hence also of Theorems~\ref{Thm1} and~\ref{Thm3}.

\section{Numerical Examples} \label{Examples}

In this section we give a number of examples illustrating both Theorem~\ref{Thm1} and Theorem~\ref{Thm2}, 
postponing the proof of the latter to Section~\ref{SmallLevel}.  We look at the five levels $N=2$, 3, 5, 6 and~7, 
giving in the first four cases (without proof) the structure of the ring $M_{*,N}=M_*(\G_0(N))$ and using 
Theorem~\ref{Thm1} to compute  the values of $R_f(X,Y)$ for a number of new forms (and for $N=2$ also for 
a couple of old forms).  We also give an example for $N=5$ showing how one can obtain the Hecke forms themselves, 
as well as their period polynomials, from the generating function, as asserted in Theorem~\ref{Thm2}, and
also an example for~$N=7$ showing that in that case the corresponding assertion is no longer true.
To keep the numerology simple, we will mostly concentrate on examples where $\dim S_{k,N}^{\rm new,\e}=1$, 
so that the polynomial $R_f(X,Y)$ for its generator~$f$ has rational coefficients, but to illustrate
the use of Theorem~2 in a non-trivial example we also look at one case ($N=5$, $k=8$, $\e=+1$)
where $S_{k,N}^{\rm new,\e}$ has dimension~2, and to show the failure of Theorem~\ref{Thm2} for~$N=7$
we also look at an example where this dimension equals~2. In the cases when $N=p$ is prime, the character
$\e\in\Dg(p)^\vee$ is determined by $\e(p)\in\{\pm1\}$, so we will write $S_{k,p}^\e$ simply as $S_{k,p}^\pm$.

\medskip\noindent{\bf N\,=\,2\,.}  The ring $M_{*,2}$ of modular forms on $\G_0(2)$ is the free algebra on 
$G_{2,2}^-(\t)$ and $G_{4,2}^-(\t)$ (or $G_4(\t)$, or $\eta(2\t)^{16}/\eta(\t)^8$), and the ideal
$S_{*,2}$ of cusp forms is the free module generated by $\D^{+}_8(\t)=\eta(\t)^8\eta(2\t)^8$.
From this information we can find the space $M_{k,2}$, and its canonical Hecke form basis, for any given weight~$k$.
In particular, in five cases the space $S^{\text{new},\pm}_{k,2}$ is one-dimensional, so that the corresponding
newform (which we denote simply by $\D_k^\pm$ rather than the more correct $\D_{k,2}^\pm$) has coefficients in~$\Z\,$: 
\begin{align*}
 \D^+_8(\t)    &\,=\, q - 8q^2 + 12q^3 + 64q^4 - 210q^5 - 96q^6 + \cdots\,, \\
 \D^-_{10}(\t) &\,=\, q + 16q^2 - 156q^3 + 256q^4 + 870q^5 - 2496q^6+\cdots\,, \\
 \D_{14}^+(\t) &\,=\, q - 64q^2 - 1836q^3 + 4096q^4 + 3990q^5 + 117504q^6 +\cdots\,, \\
 \D_{14}^-(\t) &\,=\, q + 64q^2 + 1236q^3 + 4096q^4 - 57450q^5 + 79104q^6 +\cdots\,, \\
 \D_{16}^+(\t) &\,=\, q - 128q^2 + 6252q^3 + 16384q^4 + 90510q^5 - 800256q^6 + \cdots\,.
\end{align*}
We also have two oldforms coming from $\D(\t)=q\prod(1-q^n)^{24}\in S_{12}$,
\begin{align*}
 \D_{12}^+(\t) &\,=\,\D(\t)+64\D(2\t)\,=\,  q + 40q^2 + 252q^3 - 3008q^4 + 4830q^5 +\cdots\,, \\
 \D_{12}^-(\t) &\,=\, \D(\t)-64\D(2\t)\,=\, q - 88q^2 + 252q^3 + 64q^4 + 4830q^5 + \cdots\,,
\end{align*}
as well as similar oldforms (which we will omit) with rational Fourier coefficients in weights 
16, 18, 20, 22 and 26. Computing the coefficients of the generating function~$\CC_2$ of Theorem~\ref{Thm1}, we find the 
following values of $\Rg_f(X,Y)\,=\,r_f^{\rm ev}(X)r_f^{\rm odd}(Y)/(2i)^{k-3}\la f,f\ra$ for these forms:
\smallskip
\begin{center} \renewcommand{\arraystretch}{1.3}
\begin{tabular}{|c|c|} \hline $f  $ &  $ \Rg_{f }(X,Y)$  \\ \hline 
$\D^+_8$ & $\bigl[\frac{1}{17}(8X^6-1) - (2X^4-X^2)\bigr] \cdot \bigl[(4Y^5+Y)-5Y^3\bigr]$  \\ \hline 
$\D^-_{10}$ & $ -\frac23\,\bigl[\frac{3}{31}(16X^8+1)-2(4X^6+X^2)+7X^4\bigr] \cdot \bigl[(8Y^7-Y)-7(2Y^5-Y^3)\bigr]$ \\ \hline 
$\D_{12}^+$ & $\frac18\,\bigl[\frac{33}{691}(32X^{10}-1)-(8X^8-X^2)+4(2X^6-X^4)\bigr]$\\
 & $\cdot\,\bigl[17(16Y^9+Y)-125(4Y^7+Y^3)+336Y^5\bigr]$ \\ \hline 
$\D_{12}^-$ & $\frac{1}8\,\bigl[\frac{279}{691}(32X^{10}+1)-7(8X^8+X^2)+12(2X^6+X^4)\bigr]$ \\
 & $ \cdot\,\bigl[(16Y^9-Y)-5(4Y^7-Y^3)\bigr] $ \\  \hline 
$\D_{14}^+$ & $-\frac18\,\bigl[\frac{9}{43}(64X^{12}-1)-4(16X^{10}-X^2)+11(4X^8-X^4)\bigr]$\\
 & $\cdot\,\bigl[9(32Y^{11}+Y)-55(8Y^9+Y^3)+66(2Y^7+Y^5)\bigr]$ \\ \hline 
$\D_{14}^-$ & $-\frac{11}8\,\bigl[\frac{25}{127}(64X^{12}+1)-4(16X^{10}+X^2)+15(4X^8+X^4)-32X^6\bigr]$ \\
 & $ \cdot\,\bigl[(32Y^{11}-Y)-7(8Y^9-Y^3)+18(2Y^7-Y^5)\bigr] $ \\  \hline 
$\D_{16}^{+}$ & $\frac{77}{75}\bigl[\frac{105}{257}(128X^{14}-1) - 8(32X^{12}-X^2) + 26(8X^{10}-X^4) - 39(2X^8-X^6)\bigr]$ \\
 & $\cdot\,\bigl[2(64Y^{13}+Y) - 13(16Y^{11}+Y^3) + 26(4Y^9+Y^5) -39Y^7\bigr]$ \\  \hline
\end{tabular} \renewcommand{\arraystretch}{1.0} \end{center}  
\medskip\noindent 
In each case we independently computed the values of $\la f,f\ra$ and of the coefficients of $r_f(X)$ as real numbers
and found the same results (but now only numerically, rather than exactly) in each case, the example of~$\D_8^+$ 
already having been given in Section~\ref{PerPolys}.  As a further check, one can see that the values of $\Rg_f(X,Y)$
for the oldforms $\D_{12}^\pm$ are given by
\begin{align*} \Rg_{\D_{12}^+}(X,Y) &\= \tfrac1{1152}\,\bigl(P(2X)+32P(X)\bigr)\,\bigl(Q(2Y)+32Q(Y)\bigr)\,, \\
 \Rg_{\D_{12}^-}(X,Y) &\= \tfrac1{1920}\,\bigl(P(2X)-32P(X)\bigr)\,\bigl(Q(2Y)-32Q(Y)\bigr)\,, \end{align*}
where 
$$P(X)\,=\,\tfrac{36}{691}\,\bigl(X^{10}-1)-X^2(X^2-1)^3, \;\quad Q(Y)\,=\,Y(Y^2-1)^2(Y^2-4)(4Y^2-1)$$ 
are (up to constants) the even and odd parts of the period polynomial of~$\D$, in accordance
with~\eqref{relPer} and~\eqref{PSPRatio}. (Up to a factor~$2^8$ the numbers 1152 and 1920 are 
equal to the numbers $2(2\pm 2^{-5}a_2(\Delta)+1)$ occurring in~\eqref{PSPRatio}.)

\bigskip\noindent{\bf N\,=\,3\,.} 
The ring of modular forms on $\G_0(3)$ is the ring of even polynomials in the two modular forms 
\begin{align*} \T(\t) &\= \sum_{m,\,n\,\in\,\Z}\,q^{m^2+mn+n^2}  \= 1 + 6q + 6q^3 + 6q^4 + 12q^7 + \cdots\;, \\
 H(\t) &\= 27\frac{\eta(3\t)^9}{\eta(\t)^3}  \,-\, \frac{\eta(\t)^9}{\eta(3\t)^3} 
 \= -1 + 36q + 54q^2 + 252q^3 +  \cdots \end{align*}
(both of which can also be written as Eisenstein series) of weight~1 and~3 and character~$\bigl(\frac\cdot3\bigr)$, and 
the ideal of cusp forms is again principal, generated this time by $\D_6^+=\frac1{108}(\T^6-H^2)=\eta(\t)^6\eta(3\t)^6$.  
The generators of the one-dimensional spaces $S_{k,N}^{\rm new,\pm}$ in this case are given by
\begin{align*}
\D_6^-(\t) & \,=\,  q - 6q^{2} + 9q^{3} + 4q^{4} + 6q^{5} - 54q^{6} - 40q^{7}  + \cdots\,,   \\
\D_8^+(\t) & \,=\,   q + 6q^{2} - 27q^{3} - 92q^{4} + 390q^{5} - \cdots\,,\\
\D_{10}^+(\t) & \,=\, q - 36q^{2} - 81q^{3} + 784q^{4} - 1314q^{5} + \cdots\,,  \\
\D_{10}^-(\t)  & \,=\,  q + 18q^{2} + 81q^{3} - 188q^{4} - 1530q^{5} +\cdots\,.  \end{align*}
and from Theorem~\ref{Thm1} we find that the corresponding two-variable period polynomials $\Rg_f(X,Y)$ are
\begin{center}  \renewcommand{\arraystretch}{1.3} \begin{tabular}{|c|c|}   \hline
$f$ &  $\Rg_f(X,Y)$  \\ \hline
 $\D^-_6$ & $-\frac23\,\bigl[\frac{1}{13}(9X^4+1)-X^2\bigr]\cdot\bigl[3Y^3-Y\bigr]$  \\  \hline
 $\D^+_{8}$ &  $\frac13\,\bigl[\frac2{41}(27X^6-1)-(3X^4-X^2)\bigr]\cdot\bigl[3(9Y^5+Y) - 20 Y^3\bigr]$ \\  \hline
 $\D_{10}^+$ & $ -\frac4{27}\bigl[\frac4{61}(81X^8-1)-(9X^6-X^2)\bigr]\cdot\bigl[2(27Y^7+Y)-7(3Y^5+Y^3)\bigr]$ \\ \hline
 $\D_{10}^-$ &  $- \frac{14}{27}\bigl[\frac1{11}(81X^8+1)-2(9X^6+X^2)+9X^4\bigr]\cdot\bigl[(27Y^7-Y)-8(3Y^5-Y^3)\bigr]$ \\ \hline
\end{tabular}  \renewcommand{\arraystretch}{1.0}  \end{center}  \medskip \noindent

\bigskip\noindent{\bf N\,=\,5\,.}  Here the ring $M_{*,N}$ is generated by the three Eisenstein
series $G_{2,5}^-$, $G_{4,5}^+$ and $G_{4,5}^-$, with one quadratic relation in weight~8 that we do not write down,
and the ideal of cusp forms is the principal ideal generated by the newform $\D_4^+=\eta(\t)^4\eta(5\t)^4$. 
The newforms with rational Fourier coefficients in this case are
\begin{align*}
\D_4^+(\t) & \,=\,  q - 4q^2 + 2q^3 + 8q^4  - 5q^5 - 8q^6 +6q^7  + \cdots\,,   \\
\D_6^-(\t) & \,=\,  q + 2q^2 - 4q^3 - 28q^4 + 25q^5 - 8q^6 + 192q^7 + \cdots\,,\\
\D_8^-(\t) & \,=\,  q - 14q^2 - 48q^3 + 68q^4 + 125q^5 + 672q^6 - 1644q^7 + \cdots\,.  \end{align*}
and using Theorem~\ref{Thm1} just as before we find the corresponding period polynomials
\smallskip
\begin{center}  \renewcommand{\arraystretch}{1.3} \begin{tabular}{|c|c|}   \hline
$f$ &  $\Rg_f(X,Y)$  \\ \hline
 $\D^+_4$ & $\frac4{65}\big[5X^2-1\bigr]\cdot\bigl[Y\bigr]$  \\  \hline
 $\D^-_6$ &  $-\frac65\bigl[\frac5{93}(25X^4+1)-X^2\bigr]\cdot\bigl[5Y^3-Y\bigr]$ \\  \hline
 $\D_8^-$ & $ \frac3{25}\bigl[\frac6{65}(125X^6+1)-(5X^4+X^2)\bigr]\cdot\bigl[25Y^5-Y\bigr]$ \\ \hline
\end{tabular}  \renewcommand{\arraystretch}{1.0}  \end{center}  \medskip \noindent
Of course, we can also look at eigenspaces $M_{k,N}^{\rm new,\,\e}$ that are not one-dimensional, so that the 
corresponding Hecke eigenforms no longer have rational Fourier coefficients.  The first such forms here are the form
$$ f_8 \= q \+ (10+2\sqrt{19})\,q^2 \+(10-16\sqrt{19})\,q^3 \+(48+40\sqrt{19})\,q^4 \+ \cdots $$
in $S_{8,5}^+$ and its Galois conjugate~$f_8^\sigma$, where $1\ne\sigma\in\text{Gal}(\Q(\sqrt{19})/\Q)$. If we expand 
the coefficient $C_{8,5}(X,Y,\t)$ of $T^6/6!$ in $\CC_5(X,Y,\t,T)$ and write its cuspidal part in terms of the forms
$f_8$ and $f_8^\sigma$, then we find that the two-variable period polynomial $\Rg_{f_8}(X,Y)$ is given by
\begin{align*} \Rg_{f_8}(X,Y)&\=\frac{11-7/\sqrt{19}}{375}\,\Bigl[\frac{4}{97+\sqrt{19}}(125X^6-1)-(5X^4-X^2)\bigr] \\
&\qquad \cdot\,\bigl[15(25Y^5+Y)-(137+\sqrt{19})Y^3\bigr]\,.\end{align*}

The value of $\Rg_{f_8}(X,Y)$ just given was derived assuming that the Hecke eigenforms $f_8$,~$f_8^\sigma$ were already
known. We now illustrate Theorem~\ref{Thm2} by showing how we can obtain these Hecke eigenforms as well as their
period polynomials directly from $\CC_5(X,Y,\t,T)$, {\it without} knowing them in advance.  (We chose to use this
example rather than one for $N=2$ or $N=3$ to illustrate the theorem both because it is the largest level occurring
in the theorem and because the first newforms with non-rational coefficients have smaller weight and therefore 
also smaller Fourier coefficients in this case.)  To do this, we first use Theorem~1 to compute $C_{8,5}(X,Y,\t)$
as the coefficient of $T^6$ in $\BC_5(X,Y,\t,T)$, then subtract from this its Eisenstein part as computed in Section~\ref{Eis}
to get $C_{8,5}^{\rm cusp}(X,Y,\t)$, and then symmetrize with respect to $X\leftrightarrow-X$ to get the even-odd part
$\Cg_{8,5}^{\rm cusp}(X,Y,\t)$, which is an element of $\V_{6,5}^{\rm ev}(X)\otimes\V_{6,5}^{\rm od}(Y)\otimes\Q[[q]]$.
We can then write this element with respect to the bases 
$$ (e_1^+(X), e_2^+(X), e_1^-(X), e_2^-(X)) = (125X^6-1, 5X^4-X^2,125X^6+1, 5X^4+X^2)$$
and
$$ (f_1^+(Y),f_2^+(Y),f_1^-(Y)) = (25Y^5+Y, Y^3, 25Y^5-Y)$$
of $\V_{6,5}^{\rm ev}(X)$ and $\V_{6,5}^{\rm od}(Y)$ as
$$\Cg_{8,5}^{\rm cusp}(X,Y,\t) \= \sum_{i=1}^2\sum_{j=1}^2 A_{i,j}(\t)\,e_i^+(X)\,f_j^+(Y) \+ \sum_{i=1}^2 B_{i,1}(\t)\,e_i^-(X)\,f_1^-(Y)$$
with explicit matrices $A(\t)$ and $B(\t)$ with coefficients in $\Q[[q]]$.  Writing out these matrices, we find that the 
two entries of the $2\times1$ matrix $B(\t)$ are proportional (as has to be the case), so that we obtain a factorization 
 $$ 6!\,B(\t) \= \frac3{25}\,\begin{pmatrix} \frac6{65}\\-1\end{pmatrix}\,(q - 14q^2 - 48q^3 + 68q^4 + 125q^5 + \cdots)$$
and thus recover both the Hecke form $\D_{8,5}^-$ and the value of $\Rg_{\D_{8,5}^-}(X,Y)$ as given above, without having
had to assume the value of~$\D_{8,5}^-$ to be known in advance.  Similarly, the four entries of the $2\times2$ matrix 
$A(\t)$ span a 2-dimensional space of power series. If we choose a normalized basis 
\begin{align*}  F_1(\t) &\= q \+ 0\,q^2 \+ 90\,q^3 \,-\, 152\,q^4 \,-\,  125\,q^5 \+ 192\,q^6 \+ \cdots\,, \\
 F_2(\t) & \= 0\,q \+ q^2 \,-\,  8\,q^3 \+ 20\,q^4 \+ 0\,q^5 \,-\,  70\,q^6 \+ \cdots \end{align*}
for this space, then we can write $A(\t)$ as
$$ 6!\,A(\t) \= \begin{pmatrix} \frac{1432}{39125} & -\frac{104}{313_{\vphantom y}}  \\
    -\frac{22^{\vphantom h}}{25} & 8 \end{pmatrix}\,F_1(\t) 
   \+ \begin{pmatrix} \frac{11952}{39125} & -\frac{896}{313_{\vphantom y}} \\
    -\frac{192^{\vphantom h}}{25} & 72 \end{pmatrix}\,F_2(\t)\,, $$
Writing $A(\t)$ instead as a linear combination of the two as yet unknown Hecke eigenforms 
$f_i(\t)\,=\, q+\lambda_iq^2+\cdots\,=\,F_1+\lambda_iF_2$, we find
$$ 6!\,A(\t)\,=\, \frac{a(\lambda_2)f_1+a(\lambda_1)f_2}{\lambda_1-\lambda_2}\,,\;\quad
 a(\lambda) \,=\, \begin{pmatrix} \frac{8(179\lambda-1494)}{39125} & -\frac{16(13\lambda-112)}{313_{\vphantom y}} \\
    -\frac{2(11^{\vphantom h}\lambda-96)}{25} & 16(\lambda-9)\end{pmatrix}\,.$$
The fact that the coefficients of $f_1(\t)$ and $f_2(\t)$ in $\Cg_{8,5}(X,Y,\t)$ have to factor as the product
of a polynomial in~$X$ and a polynomial in~$Y$ then tells us that $\lambda_1$ and $\lambda_2$ have to be the roots of 
the quadratic equation $\det a(\lambda) = \frac{16}{39125}\,(\lambda^2 - 20\lambda + 24)=0$, so
$\lambda_1,\,\lambda_2\=10\pm2\sqrt{19}$. We have thus obtained the Hecke forms~$f_8$ and $f_8^\sigma$ as well
as their period polynomials from Theorem~\ref{Thm1} without having to assume anything in advance.

This calculation is an example of the application of the lemma on page~461 of~\cite{Z1991}, which implies that
all the Hecke forms~$f(\t)$ are uniquely determined by the expression $\sum_f\Rg_f(X,Y)f(\t)$ if we know that the
maps $f\mapsto r_f^{\rm ev}$ and $f\mapsto r_f^{\rm ev}$ are injective.  We will show in~\S\ref{SmallLevel} that this
injectivity holds in all weights if $p=2$, 3 or 5, thus proving Theorem~\ref{Thm2}.

\bigskip\noindent{\bf N\,=\,6\,.}  Here the ring $M_{*,N}$ consists of the even polynomials in the two modular forms
$\T(\t)$ and $\T(2\t)$ of weight~1, with the same $\T(\t)$ used for~\hbox{$N=3$}.  We illustrate Theorem~\ref{Thm1} by looking
at the case~$k=6$.  The space $M_{6,6}$ is 7-dimensional, with a Hecke basis consisting of the four Eisenstein series
$G_{6,6}^{\pm,\pm}(\t)\in M_{6,6}^{\pm,\pm}$ (where we indicate a character $\e\in\Dg(6)^\vee$ by giving the pair
$(\e(2),\e(3))$), the two old forms $\L_{6,2}^\pm(\D_{6,3}^-)\in S_{6,6}^{\pm,-}$ coming from level~3, and the unique newform
$$ \D_6^{-,+}(\t) \=  q + 4q^2 - 9q^3 + 16q^4 - 66q^5 - 36q^6 + 176q^7 + \cdots\,.$$
If we decompose the coefficient $\Cg_{6,6}(X,Y,\t)$ of $T^4/4!$ in~$\Cg_6(X,Y,\t,T)$ as given by Theorem~\ref{Thm1} as
a linear combination of these seven functions with coefficients in $\hV^{\rm ev}(X)\otimes\hV^{\rm od}(Y)$, then for the
six oldforms we indeed get the values related to those for $G_4$ and $\D_{6,3}^-$ in the way discussed in~\S\ref{PerPolys},
while for the newform we find
$$ \Rg_{\D_6^{-,+}}(X,Y) \= -\frac2{63}\,\bigl[(36X^4+1) - 21X^2\bigr]\,\bigl[6Y^3-Y\bigr]\,.$$

\bigskip\noindent{\bf N\,=\,7\,.}  In this case the spaces $S_{4,7}^-$ and $S_{6,7}^+$ are 1-dimensional, generated by the forms
\begin{align*} \D_4^-(\t) &\= q - q^2-2q^3-7q^4+16q^5+2q^6-7q^7+ \cdots\,, \\  
   \D_6^+(\t) &\= q-10q^2-14q^3+68q^4-56q^5+140q^6-49q^7+ \cdots
\end{align*}
with rational coefficients, and by calculations like the ones before we find that the corresponding even-odd period polynomials 
$\Rg_f(X,Y)$ are equal to $\tfrac4{35}(7X^2-1)Y$ and $-\frac8{43\cdot 7^2}(49X^4-1)(7Y^3+Y)$, respectively.
The space $S_{6,7}^-$ is 2-dimensional, generated by the form 
$$ f_6 \= q + \tfrac{9+\sqrt{57}}2\,q^2 - 3(1+\sqrt{57})q^3 + \tfrac{5+9\sqrt{57}}2\,q^4\,+\cdots \;\in\;\Z\bigl[\tfrac{1+\sqrt{57}}2\bigr][[q]]$$
and its Galois conjugate~$f_6^\sigma$, where $1\ne\sigma\in\text{Gal}(\Q(\sqrt{57})/\Q)$, and for this form Theorem~\ref{Thm1} gives
 $$ \Rg_{f_6}(X,Y)\= \tfrac{5-3\sqrt{3/19}}{7}\Bigl[\tfrac{-17+\sqrt{3/19}}{392}(49X^4+1)+X^2\Bigr]\cdot\Bigl[7Y^3-Y\Bigr]\,.$$
Thus in this case the eigenforms $f_6$ and $f_6^\sigma$ have the same odd period polynomial up to a constant (as was indeed clear
a priori, since this polynomial must be a multiple of $7Y^3-Y$) and therefore one {\it cannot} find these two forms separately 
just by decomposing $\Cg_7^{\rm cusp}(X,Y,\t,T)$ into a sum of two products of the forms $P(X)Q(Y)f(\t)$, because this decomposition is now 
not unique. This shows that Theorem~\ref{Thm2} fails for~$p=7$, even when restricted to newforms (which was not obvious {\it a priori}, because
the dimension of $S_{k,p}^{\rm new}$ is $(p-1)k/{12}+\text O(1)$ for $p$ prime, which for $p=7$ is asymptotically the same as the dimension 
of $\V_{k-2}^{\rm ev}$ or $\V_{k-2}^{\rm od}$, so that the theorem, which had to fail for  $S_{k,7}$, might have been true for new forms; for $p>7$
the dimension even of the space of new forms is too big, so that Theorem~\ref{Thm2} cannot hold).

\bigskip

\section{Proof of Theorem 2 and Haberland-type formulas} \label{SmallLevel}

In this final section of the paper we treat the case of small prime
levels, proving Theorem~\ref{Thm2} for the primes $p=2$, 3 and 5 and
also giving in each case an expression for the scalar products $\la f,f\ra$
in terms of the periods of~$f$ for any Hecke form~$f\in S_{k,p}$. 
For convenience we abbreviate~\hbox{$w=k-2$}.

As in Section~\ref{Examples}, when the level $N$ is a prime~$p$, we identify
the elements of $\Dg(p)^\vee$ with their values on~$p$, writing $M_{k,p}^\e$ as
$M_{k,p}^\pm$ if $\e(p)=\pm1$.  We want to prove the injectivity of the period
map $r^{\rm ev/od}:f\mapsto r_f^{\rm ev/od}$ from $S_{k,p}$ to $\V_{w,p}^{\rm ev/od}$ for small values of~$p$.
(This is sufficient to prove Theorem~\ref{Thm2} since the Eisenstein part
of the generating series can be computed in advance and subtracted off by the
results of Section~\ref{Eis}.)  Since the period polynomials of
cusp forms in $S_{k,p}^+$ and $S_{k,p}^-$ take values in vector spaces
$\V_{w,p}^\pm$ that have trivial intersection, it is enough to prove the
injectivity of the restriction of $r^{\rm ev/od}$ to $ S^{ \e}_{k,p}$ for $\e=\pm1$. We also
note that $S_{k,p}^+$ can be identified with $S_k(\G_0^*(p))$ and $S_{k,p}^-$ 
with $S_k(\G_0^*(p),\chi_p)$, where $\chi_p:\G_0^*(p)\to\{\pm1\}$ is the homomorphism sending $\G_0(p)$ to~1 and 
$W_p$ to~$-1$.  (Similar statements would  apply to any squarefree level~$N$, with $S_{k,N}^\e$ for any 
$\e\in\Dg(N)^\vee$ being identified with $S_k(\G_0^*(N),\chi_\e)$ for the homomorphism $\chi_\e:\G_0(N)^*\to\{\pm1\}$
sending the coset $W_M\G_0(N)=\G_0(N)W_M$ to $\e(M)$ for all~$M\in\Dg(N)$.)

By the general Eichler-Shimura theory of periods (see e.g.~\cite{S} or \cite{K}), we know that for any 
Fuchsian group~$\G$ and any $f\in S_k(\G)$ the map $\g\mapsto r_{f,\g}\in\V_w$, where 
$r_{f,\g}(X)=\int_{\g^{-1}(\infty)}^\infty f(\t)(X-\t)^w\,d\t$, is a cocycle (one can write $r_{f,\g}(\t)$ 
with $\t\in\H$ as $\widetilde f(\t,\t)|_{-w^{\vphantom h}}(1-\g)$, with $\widetilde f$ as in~\eqref{rtildedef}, 
so \hbox{$\g\mapsto r_{f,\g}$} is a $\,\text{Hol}(\H)$-valued coboundary and hence a $\V_w$-valued cocycle), and that 
the linear map from $S_k(\G)\oplus \overline{S_{k}(\G)}$ to $H^1(\G,\V_w)$ sending $(f,\overline{g})$ to the cocycle 
$\g\mapsto r_{f,\g}{(X)}+{\overline{ r_{g,\g}(\overline X)}}$ is injective, with image 
$H_{ {\rm par}}^1(\G,  \V_w)$ (the first parabolic cohomology group of $\G$ with coefficients in $\V_w$, 
defined by cocycles sending any parabolic element $\g$ of $\G$ to an element of \hbox{$\V_w|(1-\g)$}).  
This applies not only to the trivial character, but also to the map from $S_k(\G,\chi)$ to $H^1(\G,\V_{w,\chi})$ 
for any homomorphism $\chi:\G\to\C^*$, where $\V_{w,\chi}$ is the vector space $\V_{w}$ with the twisted action 
of~$\G$ given by $P\mapsto P|_{-w^{\vphantom h},\chi}\g:=\chi(\g)\,P|_{-w^{\vphantom h}}\g$.  We use it for 
$\G=\G_0^*(p)$ and $\chi=1$ or $\chi_p$.  The last remark is that the action of the matrix $\delta=\sm-1&0\\0&1\esm$
on~$\H$ by $\t\mapsto-\bar\t$ induces an anti-linear map from $S_{k,p}^\e$
to itself by sending $f(\t)$ to $f^\delta(\t):=\overline{f(-\bar\t)}$.  Since $\delta$ fixes the imaginary axis
pointwise, a one-line calculation shows that the period polynomials of $f$ and~$f^\delta$ are related by
$\overline{r_{f^\delta}(\overline X)} =-r_f(-X)$, so if $f\in S_{k,p}^\e$ is a cusp form for which either
$r_f^{\rm ev}$ or $r_f^{\rm od}$ vanishes, then the cocycle attached to either $(f,\overline{f^\delta})$
or $(f,-\overline{f^\delta})$ by the Eichler-Shimura theorem vanishes on both $T$ and~$W_p$.

The statement of Theorem~\ref{Thm2} for  {$p=2$ or~3} is now clear, since we see from the fundamental
domains~$\mathcal F_p$ of the group $\G_0^*(p)$ as shown below that in both of these cases this group is generated 
by the two elements $T$ and $W_p$ (recall that we are always considering our matrices in $PGL_2^+(\R)$, 
so that we can write $W_p$ indiscriminately as $\sm 0&-1\\p&0\esm$ or $\frac1{\sqrt p}\sm 0&-1\\p&0\esm$), and
hence a cocycle is automatically determined by its values on these two elements. We can \phantom{use}
\smallskip


\begin{center}
\begin{tikzpicture}[scale=4]
  \def\itsqf{0.70710678118654752440084436210484903929}

  \coordinate (pm1) at (-0.5, 0.5);
  \coordinate (pp1) at ( 0.5, 0.5);
  \coordinate (p0) at (0, \itsqf);

  \draw[->] (-0.6, 0) -- (0.6, 0);
  \draw (-0.5, 0) -- (-0.5, -0.025); 
   \node at (-0.5, -0.020) [anchor=north] {$-\frac{1}{2}$};
  \draw (0, 0) -- (0, -0.025); 
   \node at (0, -0.020) [anchor=north] {$\vphantom{\frac12}0$};
  \draw (0.5, 0) -- (0.5, -0.025); 
   \node at (0.5, -0.020) [anchor=north] {$\frac{1}{2}$};

  \draw[fill=black!20!white] (-0.5,1.2) -- (pm1) 
  arc (135:45:\itsqf) -- (0.5,1.2);


  \fill (pm1) circle (0.005);
  \fill (p0) circle (0.005);
  \fill (pp1) circle (0.005);

  \draw[<->,shift={(pm1)}] (0,-0.04)+(90+45:0.03) arc (90+45:450-45:0.03);
  \draw[<->,shift={(p0)}]  (0,-0.04)+(90+45:0.03) arc (90+45:450-45:0.03);
  \draw[<->,shift={(pp1)}] (0,-0.04)+(90+45:0.03) arc (90+45:450-45:0.03);

  \node at (pm1) [anchor=west] {$\ P_{-1}$}; 
  \node at (p0) [anchor=south] {$P_{0}$}; 
  \node at (pp1) [anchor=east] {$P_{1}\,$}; 

  \node at (0, 1.0) {$\mathcal{F}_2$};

  \node at ([shift={(0,-0.07)}]pm1) [anchor=north]  {$\widetilde U_2$};
  \node at ([shift={(0,-0.07)}]p0) [anchor=north]  {$W_2$};
  \node at ([shift={(0,-0.07)}]pp1) [anchor=north]  {$U_2$};
\end{tikzpicture}
\hspace{9mm}
\begin{tikzpicture}[scale=4]
  \def\itsqf{0.28867513459481288225457439025097872782}

  \coordinate (pm1) at (-0.5, \itsqf);
  \coordinate (pp1) at ( 0.5, \itsqf);
  \coordinate (p0) at (0, 2*\itsqf);

  \draw[->] (-0.6, 0) -- (0.6, 0);
  \draw (-0.5, 0) -- (-0.5, -0.025); 
   \node at (-0.5, -0.020) [anchor=north] {$-\frac{1}{2}$};
  \draw (0, 0) -- (0, -0.025); 
   \node at (0, -0.020) [anchor=north] {$\vphantom{\frac12}0$};
  \draw (0.5, 0) -- (0.5, -0.025); 
   \node at (0.5, -0.020) [anchor=north] {$\frac{1}{2}$};

  \draw[fill=black!20!white] (-0.5,1.2) -- (pm1) 
  arc (150:30:2*\itsqf) -- (0.5,1.2);


  \fill (pm1) circle (0.005);
  \fill (p0) circle (0.005);
  \fill (pp1) circle (0.005);

  \draw[<->,shift={(pm1)}] (0,-0.04)+(90+45:0.03) arc (90+45:450-45:0.03);
  \draw[<->,shift={(p0)}]  (0,-0.04)+(90+45:0.03) arc (90+45:450-45:0.03);
  \draw[<->,shift={(pp1)}] (0,-0.04)+(90+45:0.03) arc (90+45:450-45:0.03);

  \node at (pm1) [anchor=west] {$\ P_{-1}$}; 
  \node at (p0) [anchor=south] {$P_{0}$}; 
  \node at (pp1) [anchor=east] {$P_{1}\,$}; 

  \node at (0, 1.0) {$\mathcal{F}_3$};

  \node at ([shift={(0,-0.07)}]pm1) [anchor=north]  {$\widetilde U_3$};
  \node at ([shift={(0,-0.07)}]p0) [anchor=north]  {$W_3$};
  \node at ([shift={(0,-0.07)}]pp1) [anchor=north]  {$U_3$};
\end{tikzpicture}

 \end{center}

\smallskip

\noindent   
use the form of the fundamental domains to give a more complete result in these two cases, describing the  {images of
the even and odd period maps} rather than merely proving their injectivity.  The vertices of~$\mathcal F_2$ are $\infty$, 
$P_{\pm1}=\pm\frac12+\frac i2$ and $P_0=\frac i{\sqrt 2}$  and those of~$\mathcal F_3$ are $\infty$, 
$P_{\pm1}=\pm\frac12+\frac i{\sqrt{12}}$ and $P_0=\frac i{\sqrt 3}$, where in each case $\infty$~is fixed by~$T$, 
$P_0$ by $W_p$, $P_1$ by $U_p:=TW_p=\sm p&-1\\p&0\esm$, and $P_{-1}$ by $\widetilde U_p:=T^{-1}W_p=\sm-p&-1\\p&0\esm$.  
In particular, the group $\G_0^*(p)$ is generated by the two elements 
$W_p$ and $U_p$ with the relations~$W_p^2=U_p^{\,2p}=1$ (once again, in $PGL_2^+(\R)$). It follows that the 
period polynomial of any $f\in S_{k,p}^\e$ belongs to the subspace
$$ \W_{w,p}^\e \df \Bigl\{ P\in \V_{w,p}^\e \;:\; \sum_{j=0}^{2p-1} P \bigl|_{-w^{\vphantom h},\e}U_p^j=0\Bigr\}$$
of $\V_{w,p}^\e$ (here $|_{-w,\e}$ means $|_{-w,\chi_\e}$ with $\chi_+=1$, $\chi_-=\chi_p$), because
$$ \Bigl(r_f\Bigl|_{-w,\e^{\vphantom h}}\sum_{j=0}^{2p-1} U_p^j\Bigr)(X) 
 \= \sum_{j=0}^{2p-1} \int_{U_p^{-j-1}(\infty)}^{U_p^{-j}(\infty)}f(\t)\,(\t-X)^{w}\,d\t \=0\,.$$
We can also consider the even and odd period maps $f\mapsto r_f^{\rm ev/od}\in \W_{w,p}^{\rm ev/od,\e}$,
and note that, as in~Section~\ref{PerPolys}, these maps extend in the even case to all of $M_{k,p}^\e$, whereas
for the odd period map we would have to replace $\W_{w,p}^\e$ by a one dimension larger space $\widehat\W_{w,p}^\e$
in order to be able to include the Eisenstein series in~$M_{k,p}^\e$.  By the Eichler-Shimura theorem once again,
we know that the space $H^1_{\rm par}(\G_0^*(p),\V_{w,\chi_\e})$ of parabolic cohomology classes
on~$\G_0^*(p)$ with values in~$\V_{w,\chi_\e}$ is isomorphic to two copies of $S_{k,p}^\e$  (or more correctly, of 
one copy of this space and one copy of its complex conjugate), and from the presentation of $\G_0^*(p)$ we see that 
this space is isomorphic to \hbox{$\W_{w,p}^{ \e}/ \langle \e p^{w/2} X^w-1\rangle$}. Putting everything together, 
we have:
\begin{proposition}\label{2and3} For $p\in\{2,3\}$ and $\e\in\{\pm1\}$ the maps $f\mapsto r_f^{\rm ev}$ and
$f\mapsto r_f^{\rm od}$ give isomorphisms
$$ M_{k,p}^\e \stackrel{\sim}\longrightarrow \W_{w,p}^{\rm ev,\e}\,, \qquad 
 S_{k,p}^\e \stackrel{\sim}\longrightarrow \W_{w,p}^{\rm od,\e}\,. $$
\end{proposition}
We mention in passing that the translation into generating functions of the fact that the period polynomials belong 
to~$\widehat\W_{w,p}$, together with the equality~$\CC_p=\BC_p$ of Theorem~\ref{Thm1}, gives the 4-term theta relation 
\begin{align*} & \BC_2(X,Y,\t,T) \m  \BC_2\Bigl(\frac X{1-2X},Y,\t,(1-2X)T\Bigr)  \\
  & \quad \+ \BC_2\Bigl(\frac{X-1}{2X-1},Y,\t,(2X-1)T\Bigr) \m  \BC_2(X-1,Y,\t,T) \= 0\,, \end{align*}
and a similar 6-term relation for~$\BC_3(X,Y,\t,T)$.  

To complete the above result, we also give a formula expressing the Petersson scalar product of two cusp forms $f,\,g\in S_{k,p}^\e$ 
($p=2$ or~3) in terms of the period polynomials $r_f$ and $r_g$.  Formulas of this type were given by Haberland~\cite{H} and
in a slightly different form in~\cite{KZ}, for the group $SL_2(\Z)$ and arbitrary weight, and for general Fuchsian groups in~\cite{Z1985} 
(only for weight~2, but the corresponding formula holds in all weights).  But of course for general groups one needs
the full cocycle $\gamma\mapsto r_{f,\g}$.  Here we give the explicit formulas for~$p=2$ and~$p=3$, and also for~$p=5$ below, giving
only a sketch in each case since the methods are by now standard and since an equivalent result is also stated in the recent 
paper~\cite{PP2013} by Pasol and~Popa. (But they do not give any proof or reference for the cases $p=3$ and $p=5$,
and also express everything in terms of the standard generators of $SL_2(\Z)$ and the inclusion $\G_0(p)\subset SL_2(\Z)$, while 
we work directly with~$\G_0^*(p)$ and its generators.)
We recall the standard definition of a $PGL_2^+(\R)$-invariant scalar product on~$\V_w$, given by the formula
$(X^r,X^s)_{\V_w}=(-1)^r\delta_{r+s,w}/\binom wr$ for $0\le r,\,s\le w$, or equivalently by the formula 
$((X-a)^w,P(X))_{\V_w}=P(a)$ for~$a\in\C$ and~$P\in\V_w$.  
\begin{proposition}\label{Hab2and3} For $p\in\{2,3\}$ and $\e\in\{\pm1\}$, the Petersson scalar product
of two cusp forms $f,\,g\in S_{k,p}^\e$ is given by
\be\label{Hab23} (2i)^{k-1}\,\la f,\,g\ra \=  -\frac1{p}\,\bigl(r_f|_{-w,\e}\mathcal A_p,\,r_g^\iota\bigr)_{\V_w}\,, \ee
where $\,\mathcal A_p:=\sum_{j=1}^{p-1}(p-j)(\widetilde U_p^j-U_p^j)\in\Z[\G_0^*(p)]\,$ and $r_g^\iota(X):=\overline{r_g\bigr(\overline X\bigl)}$.
\end{proposition} 
We make several remarks about this proposition before giving its proof. \newline
\noindent{\bf 1.} Equation~\eqref{Hab23} is compatible with the equation $\overline{\la f,g\ra}=\la g,f\ra$,
because the operator $\,|\mathcal A_p$ is anti-self-adjoint with respect to $(\;,\;)_{\V_w}$.  \newline
\noindent{\bf 2.} Since $r_g^\iota(X)=-r_{g^\delta}(-X)$, and since  $\,|\mathcal A_p$ also anti-commutes with the matrix $\delta=\sm-1&0\\0&1\esm$
(because $\delta U_p\delta=\widetilde U_p$),  we can rewrite $2\bigl(r_f|\mathcal A_p,\,r_g^\iota\bigr)_{\V_w}$ as
$\bigl(r_f^{\rm ev}|\mathcal A_p,\,r_{g^\delta}^{\rm od}\bigr)_{\V_w}+\bigl(r_{g^\delta}^{\rm ev}|\mathcal A_p,\,r_f^{\rm od}\bigr)_{\V_w}$,
which is symmetric in $f$ and~$g^\delta$~and involves only pairings between even and odd period polynomials, as it should.  This immediately
implies the injectivity of $ r^{\rm ev/od}: f\mapsto r_f^{\rm ev/od}$ if $f$  has real Fourier coefficients (because with $f=f^\delta=g$ 
it shows that $(f,f)$ vanishes if $r_f^{\rm ev}$ or $r_f^{\rm od}$ vanishes), and this in turn implies the injectivity for all~$f$.  
(If $f=f_1+if_2$ with $f_1$ and $f_2$ real, then $r_f^{\rm ev}$ is the sum of $r_{f_1}^{\rm ev}$ and $ ir_{f_2}^{\rm ev}$, one of which has 
real coefficients and the other pure imaginary coefficients, so if $r_f^{\rm ev}=0$ then $r_{f_1}^{\rm ev}=r_{f_2}^{\rm ev}=0$ and 
hence $f_1=f_2=0$, and similarly for $r_f^{\rm od}$.) \newline
\noindent{\bf 3.} It also follows that $\Phi[\Rg_f(X,Y)]=1$ for all~$f\in\BB_{k,p}^{\rm cusp}$, where
$\Phi$ is the map from $\V^{\rm ev,\e}_{w,p}(X)\otimes\V^{\rm od,\e}_{w,p}(Y)\subset\C[X,Y]$ to~$\C$ sending
$r_1(X)r_2(Y)$ to $\frac1{2p}(r_1|\mathcal A_p,r_2)$. (One can check this numerically for each of the 11 cusp forms in the tables 
for~$N=2$ and~$N=3$ in~\S\ref{Examples}.) This implies that $\Phi[C_k(X,Y,\t)]$ is the modular form whose $n$th Fourier
coefficient is the trace of $T_n$ on $M_{k,p}$, so one could also use the proposition to compute the
traces of Hecke operators for $\G_0^*(2)$ and $\G_0^*(3)$, as was done in~\cite{Z1990} for the full modular group. \newline
\noindent{\bf 4.} Finally, one could also use Proposition~\ref{Hab2and3} to give an explicit description of the codimension~1 subspace 
$r^{\rm ev}(S_{k,p}^\e)\subset\W_{k,p}^{\rm ev,\e}$, as was done in~\cite{KZ} for~$SL_2(\Z)$, by extending~\eqref{Hab23} 
to the case when one of~$f$ or~$g$ is an Eisenstein series and observing that $\la f,G_{k,p}^\e\ra$ vanishes for $f\in S_{k,p}^\e$.

\begin{proof} Let $\widetilde f(X,\t)$ be the function defined in~\eqref{rtildedef}, and $\widetilde g(X,\t)$ the corresponding
function for~$g$.  Then the function $G(\t):=\overline{\widetilde g(\bar\t,\t)}$ transforms by $G|_{-w,\e}(1-\g)=\overline{r_{g,\g}(\bar\t)}$
for all $\g\in\G_0^*(p)$.  Also, $\partial G/\partial\bar\t=-(2iy)^w\overline{g(\t)}$, so
\begin{align*} (2i)^{k-1}\,\la f,\,g\ra &\= \iint_{\mathcal F_p} d\bigl[f(\t)G(\t)d\t\bigr] \= \int_{\partial\mathcal F_p} f(\t)\,G(\t)\,d\t \\
   & \= \int_{P_{-1}}^{P_1} f(\t)\,G(\t)\,d\t \= \frac12\,\int_{P_{-1}}^{P_1} f(\t)\,r_g^\iota(\t)\,d\t\,, \end{align*}
where the second and third equalities follow from Stokes's theorem and the periodicity of $f(\t)G(\t)$ and the last one because the
lower edge of~$\mathcal F_p$ is mapped orientation-reversingly onto itself by~$W_p$ and~$f|_{k,\e}W_p=f$. By the above-mentioned 
property of $(\,\cdot\,,\,\cdot\,)_{\V_w}$ we have $\,r_g^\iota(\t)=((X-\t)^w,r_g^\iota(X))_{\V_w}$, so this can be rewritten
\begin{align*} (2i)^{k-1}\,\la f,\,g\ra &\=  \biggl(\frac12\,\int_{P_{-1}}^{P_1} f(\t)\,(\t-X)^w\,d\t\,,\,r_g^\iota(X)\biggr)_{\V_w} \\
   & \= \frac12\,\Bigl(\widetilde f(X,P_{-1})\m \widetilde f(X,P_1),\,r_g^\iota(X)\Bigr)_{\V_w}\;. \end{align*}
Since $r_g^\iota\bigr|_{-w,\e}\sum_{\text{$j$ (mod~$2p$)}}U_p^j=0$ and  $\widetilde f(X,P_1)|_{-w,\e}(1-U_p^j)=r_{f,U_p^j}$
for all~$j$ (by the transformation law of~$\widetilde f$ and because $U_p$ fixes~$P_1$), we have
\begin{align*} \Bigl(\widetilde f(X,P_1),\,r_g^\iota(X)\Bigr)_{\V_w} &\= 
\Bigl(\widetilde f(X,P_1)\bigr|_{-w,\e}\Bigl(1-\frac1{2p}\sum_{\text{$j$ (mod~$2p$)}}U_p^j\Bigr),\,r_g^\iota(X)\Bigr)_{\V_w} \\
 &\=\frac1{2p}\sum_{\text{$j$ (mod~$2p$)}}\bigl(r_{f,U_p^j}(X),\,r_g^\iota(X)\bigr)_{\V_w} \,. \end{align*}
Also, the cocycle property gives $r_{f,U_p^j}=r_f|_{-w^{\vphantom h},\e}(1+U_p+\cdots+U_p^{j-1})$ (and similarly for~$\widetilde U_p$).
The assertion of the proposition now follows after a short calculation whose details are left to the reader. 
\end{proof}

We now turn to the remaining case $p=5$.  Here the situation is more complicated because the element $TW_5$ no
longer has finite order and because $\G_0^*(5)$ is no longer generated by $W=W_5$ and $T$.  Instead, we have
\be\label{gens} \G_0^*(5) \= \big\la A,\,B,\,W\big\ra\,,  \qquad A^2=B^2=W^2=1,\quad BAW\,=\,T\,,\ee
where $A=\sm2&-1\\5&-2\esm$, $B=\sm5&-3\\10&-5\esm$.  One can see this from the fundamental domain for~$\G$ shown below,
which is a hyperbolic polygon whose vertices $P_1=\frac{2+i}5$, $P_2=\frac{\sqrt5+i}{2\sqrt5}$, $P_0=\frac i{\sqrt5}$,
$P_{-1}=\frac{-2+i}5$ and $P_{-2}=\frac{-\sqrt5+i}{2\sqrt5}$ and~$\infty$ are the fixed points of the five involutions 
$A$, $B$, $W$, $A_1=WAW$ and $B_1=T^{-1}BT$ of $\G_0^*(5)$ and of the parabolic element~$T$.  What we have to prove 
is that for any $f\in S_{k,5}^\e\cong S_k(\G_0^*(5),\chi_\e)$ the cocycle  $\g\mapsto r_{f,\g}$ is determined by 
just the even or just the odd part of the basic period polynomial~$r_{f^{\vphantom h}}=r_{f^{\vphantom h},W}$.
\vadjust{\rlap
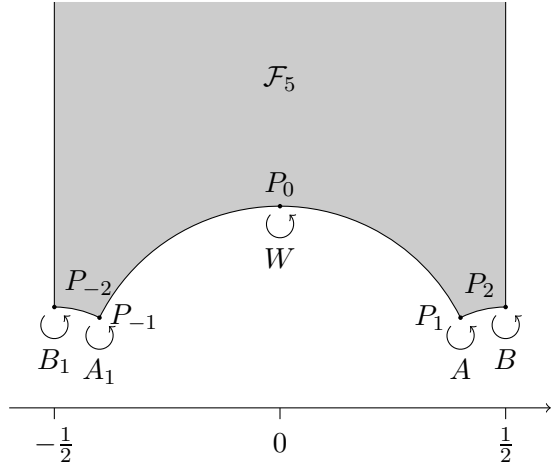
{%
\centerline{%
\begin{tikzpicture}[scale=6]
  \draw[white] (-0.6,-0.12) -- (-0.6,0.93);
  \def\itsqf{0.22360679774997896964091736687312762354}
  \def\angf{63.434948822922010648427806279546705329}
  %
  \coordinate (pm2) at (-0.5, \itsqf);
  \coordinate (pp2) at ( 0.5, \itsqf);
  \coordinate (pm1) at (-0.4, 0.2);
  \coordinate (pp1) at ( 0.4, 0.2);
  \coordinate (p0) at (0, \itsqf + \itsqf);
  %
  \draw[->] (-0.6, 0) -- (0.6, 0);
  \draw (-0.5, 0) -- (-0.5, -0.025); 
   \node at (-0.5, -0.020) [anchor=north] {$-\frac{1}{2}$};
  \draw (0, 0) -- (0, -0.025); 
   \node at (0, -0.020) [anchor=north] {$\vphantom{\frac12}0$};
  \draw (0.5, 0) -- (0.5, -0.025); 
   \node at (0.5, -0.020) [anchor=north] {$\frac{1}{2}$};
   %
  \draw[fill=black!20!white] (-0.5,.9) -- (pm2) 
  arc (90:\angf:\itsqf) 
  arc (90+\angf:90-\angf:\itsqf+\itsqf) 
  arc (180-\angf:90:\itsqf) -- (0.5,.9);
  %
  %
  \fill (pm2) circle (0.005);
  \fill (pm1) circle (0.005);
  \fill (p0) circle (0.005);
  \fill (pp1) circle (0.005);
  \fill (pp2) circle (0.005);
  %
  \draw[->,shift={(pm2)}] (0,-0.04)+(90+45:0.03) arc (90+45:450-45:0.03);
  \draw[->,shift={(pm1)}] (0,-0.04)+(90+45:0.03) arc (90+45:450-45:0.03);
  \draw[->,shift={(p0)}]  (0,-0.04)+(90+45:0.03) arc (90+45:450-45:0.03);
  \draw[->,shift={(pp1)}] (0,-0.04)+(90+45:0.03) arc (90+45:450-45:0.03);
  \draw[->,shift={(pp2)}] (0,-0.04)+(90+45:0.03) arc (90+45:450-45:0.03);
  %
  \node at (pm2) [anchor=south west] {$P_{-2}$}; 
  \node at (pm1) [anchor=west] {$P_{-1}$}; 
  \node at (p0) [anchor=south] {$P_{0}$}; 
  \node at (pp1) [anchor=east] {$P_{1}\,$}; 
  \node at (pp2) [anchor=south east] {$P_{2}$}; 
  \node at (0, 0.73) {$\mathcal{F}_5$};
  \node at ([shift={(0,-0.07)}]pm2) [anchor=north]  {$B_{1}$};
  \node at ([shift={(0,-0.07)}]pm1) [anchor=north]  {$A_{1}$};
  \node at ([shift={(0,-0.07)}]p0) [anchor=north]  {$W$};
  \node at ([shift={(0,-0.07)}]pp1) [anchor=north]  {$A$};
  \node at ([shift={(0,-0.07)}]pp2) [anchor=north]  {$B$};
\end{tikzpicture}}}}

As in the cases $p=2$ and $p=3$, the Eichler-Shimura theory of periods tells us that the parabolic cohomology 
group $H^1_{\rm par}(\G_0^*(5),\V_{w,\chi_\e})$ is isomorphic to two copies of $S_{k,5}^\e$.  Since $\G_0^*(5)$ has 
only one cusp, all of its parabolic elements are conjugate to~$T$, so parabolic cohomology classes can be represented by 
cocycles $\g\mapsto r_{\g}$ with~$r_T=0$, these representative being unique up to the 1-dimensional space of coboundaries
given by $\g\mapsto1|(1-\g)$. (Here and for the rest of the section we write simply $\,|\,$ for the twisted operation
 $\;|_{-w,\e}=|_{-w,\chi_\e}$ of~$\G_0^*(5)$ on~$\V_w$.)  From the above presentation of $\G_0^*(5)$, a cohomology class
is determined by the three elements~$r_A$, $r_B$ and $r=r_W$, with
 $r_A|(1+A)= r_B|(1+B)=r|(1+W)=0$ (because $A$, $B$ and~$W$ are involutions) and $r_B|AW+r_A|W+r=r_{BAW}=0$, 
so we can eliminate $r_A$ and identify the space $Z^1_{\rm par}(\G_0^*(5),\V_{w,\chi_\e})$ of parabolic cocycles with the space
$\W_{w,5}^\e $ of pairs of polynomials $(r,\,r_B)$ in $(\V_{w,5}^{\e})^2$ satisfying $r|(1+W)=r_B|(1+B)=(r-r_B)|(1+A)=0$. If it were 
true that such a pair is determined by its first element~$r$, then we would be done.  However, this is not the case.
Instead, if $r=0$ we have $r_B=-r_B|A=r_B|BA$.  Since $BA$ is hyperbolic, its fixed-point set in~$\V_w$ is 1-dimensional, spanned
by the polynomial $(5X^2-5X+1)^{w/2}$, and this polynomial is also anti-invariant (with respect to the twisted action of~$\G_0^*(5)$)
under both~$A$ and~$B$ if $\e=-(-1)^{w/2}=(-1)^{k/2}$.  The map $(r,r_B)\mapsto r$ from parabolic cocycles to polynomials is
therefore {\it not} injective in these cases, but has a 1-dimensional kernel.  But this failure does not mean that the map $f\mapsto r_f$ 
is not injective, because the image of $(S_{k,5}^\e)^2$ under the period map has codimension~1 in the space of all parabolic cocycles. and 
the offending vector $(r,r_B)=(0,(5X^2-5X+1)^{w/2})$ luckily does not belong to this image.  This follows from a Haberland-type formula, 
which we now state briefly, for the Petersson scalar products of cusp forms on~$\G_0(5)$ in terms of their periods.

To state the formula, we need to use the decomposition of the space of cocycles into an even and an odd part, corresponding to the
$(\pm1)$-eigenspaces of the action of the involution~$\delta$ of the space $H^1_{\rm par}$ induced by the involution
$\t\mapsto-\bar\t$ of $\H/\G_0^*(5)$.  For this it is convenient to change our description of the space $\W_{w,5}^\e$, 
replacing $r_B(X)$ by $r^*(X):=r_B((X+1)/2)$, because then the integral representation 
$$ r_f^*(X) \= r_{f,B}\Bigl(\frac{X+1}2\Bigr) \= 2^{1-k}\int_0^\infty f\Bigl(\frac{\t+1}2\Bigr)\,(X-\t)^w\,d\t $$
and the same one-line calculation as for $r_f$ show that $\overline{r^*_{f^\delta}(\overline X)} =-r^*_f(-X)$. We have isomorphisms
$$ \r^{\rm od}:\,S_{k,5}^\e \overset\sim\longrightarrow \W_{w,5}^{\rm od,\e}, \quad
\r^{\rm ev}:\,S_{k,5}^\e \overset\sim\longrightarrow \W_{w,5}^{\rm ev,\e}/\la(1-\e5^{w/2}X^w,1-\e5^{w/2}X^w)\ra,$$
given by $f\mapsto(r_f^{\rm ev/od},(r_f^*)^{\rm ev/od})$, where 
 $$  \W_{w,5}^{\rm ev/od,\e} \= \bigl\{(r,r^*)\in(\V_{w,5}^{\rm ev/od,\e})^2\mid (r(X)-r^*(2X-1))|(1+A)=0\}\,.$$
We illustrate this with the examples $k=8$, $\e=\pm1$.  For $\e=-1$ we find that a pair of polynomials
$(r,r^*)=(a_0(125X^6+1)+a_2(5X^4+X^2)+a_1(25X^5-X)$, $a_0^*(125X^6+1)+a_2^*(5X^4+X^2)+a_1^*(25X^5-X))$ belongs to $\W_{6,5}^-$ 
if and only if $(a_0^*,a_2^*,a_1^*)=(8a_0+\frac{33}{20}a_2,-\frac{17}4a_2,-6a_1)$, so that $r^*$ is completely determined by~$r$, while
for $\e=1$ we find that $(r,r^*)=(a_0(125X^6-1)+a_2(5X^4-X^2)+a_1(25X^5+X)+a_3X^3,a_0^*(125X^6-1)+a_2^*(5X^4-X^2)+a_1^*(25X^5+X)+a_3^*X^3)$
belongs to $\W_{6,5}^+$ if and only if $(a_0^*,a_2^*,a_1^*,a_3^*)=(-8a_0+m,-8a_2-15m$, $\frac{149}4a_1+\frac{15}4a_3,-\frac{625}2a_1-\frac{67}2a_3)$ 
for some~$m$, so that in this case $r^*$ is determined by~$r$ only up to the addition of a multiple of $(5X^2-1)^3$, in accordance with the 
assertions above.  Both statements can be checked numerically for the period polynomials of the cusp forms $\D_8^-\in S_{8,5}^-$, 
where $a_2=-\frac{65}6a_0$, and $f_8\in S_{8,5}^+$, where $a_2=-\frac{97+\sqrt{19}}4a_0$ and $a_3=-\frac{137+\sqrt{19}}{15}a_1$ as 
given in the $N=5$ tables in Section~\ref{Examples}.

The formula for the scalar product of two cusp forms $f,\,g\in S_{k,5}^\e$ with real Fourier coefficients can now be stated as
\be \label{Hab5} (2i)^{k-1}\,\la f,\,g\ra \= \bigl(r_f,\,r_g^*|(T-T^{-1})\bigr)_{\V_w} \+ \bigl(r_g,\,r_f^*|(T-T^{-1})\bigr)_{\V_w}\,. \ee
We omit the proof of this formula, which proceeds along exactly the same lines as the proof of Proposition~\ref{Hab2and3} above, but 
mention that it implies the injectivity of the maps $f\mapsto r_f^{\rm ev}$ and $f\mapsto r_f^{\rm od}$ by the same argument as in the 
second remark following that proposition (now using that $|(T-T^{-1})$ is anti-self-adjoint and anti-commutes with $\delta$). We can 
also check using the data given in the tables for~$N=5$ in Section~\ref{Examples} that~\eqref{Hab5} gives the correct values of $\la f,f\ra$ for 
each of the cusp forms~$f$ listed there and also gives the orthogonality of the Hecke form $f_8\in S_{8,5}^+$ with its Galois conjugate eigenform.

\bigskip

{\bf Acknowledgments.} The first autohr was partially supported by the grants   NRF-2017R1A2B2001807 and
NRF-2016R1A2B1012330. The second author was partially supported by the grant
NRF-2017R1D1A1B03029519 and NRF-2009-0093827.
  The first author would like to thank the Max Planck 
Institute for Mathematics in Bonn, and the third author both POSTECH (Pohang) and KIAS (Seoul),
for supporting research stays while this work was being completed.
 
\bigskip \bigskip 

\bibliographystyle{amsplain}

\end{document}